\documentclass{amsart}

\usepackage{a4wide}
\usepackage{amscd}
\usepackage{amssymb}
\usepackage{amsthm}
\usepackage{amsmath}
\usepackage{times}
\usepackage{latexsym}
\usepackage{amsfonts}
\usepackage{graphicx}
\usepackage{graphicx, subfigure}
\usepackage{hyperref} 

\newtheorem{theorem}{Theorem}
\theoremstyle{plain}

\newtheorem{assumption}{Assumption}

\newtheorem{corollary}[theorem]{Corollary}

\newtheorem{definition}[theorem]{Definition}

\newtheorem{lemma}[theorem]{Lemma}

\newtheorem{proposition}[theorem]{Proposition}

\numberwithin{equation}{section}
\numberwithin{theorem}{section}

\newcommand{\R}{\ensuremath{\mathbb{R}}}
\newcommand{\E}{\ensuremath{\mathbb{E}}}
\newcommand{\N}{\ensuremath{\mathbb{N}}}

\newcommand{\V}{\ensuremath{\mathcal{V}}}
\newcommand{\U}{\mathcal{U}}
\newcommand{\Hi}{\mathcal{H}}
\newcommand{\Le}{\mathcal{L}}

\def\e{{\mathrm{e}}}
\allowdisplaybreaks

\title{Representation and approximation of ambit fields in Hilbert space}
\author[Benth]{Fred Espen Benth}
\address[Fred Espen Benth]{\\
Department of Mathematics \\
University of Oslo\\
P.O. Box 1053, Blindern\\
N--0316 Oslo, Norway \\
and \\
Centre for Advanced Study \\
Drammensveien 78 \\
N-0271 Oslo, Norway}
\email[]{fredb\@@math.uio.no}
\urladdr{http://folk.uio.no/fredb/}
\author[Eyjolfsson]{Heidar Eyjolfsson}
\address[Heidar Eyjolfsson]{\\
Department of Mathematics \\
University of Bergen \\
P.O. Box 7803 \\
N--5020 Bergen, Norway.}
\email[]{Heidar.Eyjolfsson\@@uib.no}
\urladdr{http://www.uib.no/en/persons/Heidar.Eyjolfsson}

\date{\today}
\thanks{F. E. Benth acknowledges financial support from the research projects "Managing Weather Risk in Energy Markets (MAWREM)" and "Finance, Insurance, Energy, Weather and Stochastics (FINEWSTOCH)",
both funded by the Norwegian Research Council. H. Eyjolfsson acknowledges financial support from Finansmarkedsfondet.}
\begin{document}
\maketitle
\begin{abstract}
We lift ambit fields as introduced by Barndorff-Nielsen and 
Schmiegel~\cite{BN-Schm} to a class of Hilbert space-valued volatility modulated Volterra processes. 
We name this class Hambit fields, and show that they can be expressed as a 
countable sum of weighted real-valued volatility modulated Volterra processes. Moreover, Hambit fields
can be interpreted as the boundary of the mild solution of a certain first order stochastic partial 
differential equation. This stochastic partial differential equation is formulated on a suitable 
Hilbert space of functions on the positive real line with values in the state space of the Hambit field. 
We provide an explicit construction of such a space. Finally, we apply this interpretation of 
Hambit fields to develop a finite difference scheme, for which we prove convergence under
some Lipschitz conditions.
\end{abstract} 

\section{Introduction}

Ambit fields, introduced by Barndorff-Nielsen and Schmiegel~\cite{BN-Schm}, have attracted
much attention in recent years being a powerful tool to model stochastic phenomena like 
turbulence, tumor growth, weather dynamics,  and financial prices
(see Barndorff-Nielsen and Schmiegel~\cite{BN-Schm}, Barndorff-Nielsen, Benth and Veraart~\cite{BNBV-spot,BNBV-forward}, Benth and \v{S}altyt\.e Benyth~\cite{BSB-book}, Corcuera et al.~\cite{CFSV}
and Vedel Jensen et al.~\cite{VJ}). The class of ambit fields 
is analytically tractable, and provides a framework for a probabilistic description of the 
dynamics of noisy systems which are more general than the conventional stochastic partial differential
equations (see Barndorff-Nielsen, Benth and Veraart~\cite{BNBV-first}).  

Following Barndorff-Nielsen and Schmiegel~\cite{BN-Schm}, an ambit field
is defined as a real-valued random field on $\R_+\times\R^d$ and a filtered probability space 
$(\Omega,\mathcal{F},\{\mathcal{F}_t\}_{t\geq 0}, P)$ of the form
\begin{equation}
\label{def-ambit-classic}
Z(t,x)=\int_0^t\int_{A}g(t,s,x,y)\sigma(s,y)\, L(dy,ds)\,.
\end{equation}
Here, $(t,x)\in\R_+\times A$, $A\subset\R^d$ is a Borel measurable subset called
the {\it ambit set}, $g$ a measurable real-valued function on $\R_+\times\R_+\times\R^d\times\R^d$ and
$\sigma$ a real-valued predictable random field on $\R_+\times\R^d$. The function $g$ is sometimes referred
to as the {\it kernel function}, and $\sigma$ is modelling the volatility or intermittency. Finally,
$L$ is a L\'evy basis, where $\sigma$ and $L$ are assumed independent. In this paper we restrict
our attention to $L$ being a square-integrable L\'evy basis. Moreover, we suppose 
$L$ to have mean zero. Using the integration
concept of Walsh (see Walsh~\cite{Walsh}), the ambit field $Z(t,x)$ in \eqref{def-ambit-classic} is
well-defined if
\begin{equation}
\label{int-cond-ambit-classical}
\int_{[0,t]\times A}g^2(t,s,x,y)\E[\sigma^2(s,y)]\text{Var}(L'(y,s))\,c(dy,ds)<\infty\,, 
\end{equation}
where $c$ is the control measure and $L'$ the L\'evy seed associated with $L$. Indeed,
$\text{Var}(L'(x,t))\,c(dx,dt)$ is equal to the Radon-Nikodym derivative of the covariance measure
of $L$.   
We refer to Barndorff-Nielsen and Schmiegel~\cite{BN-Schm} or,
the more recent survey paper of Barndorff-Nielsen, Benth and Veraart~\cite{BNBV-banach} for details and discussions about ambit fields and their properties
and applications. An analysis on stochastic integration for random fields as introduced by Walsh
applied to ambit fields can be found in Barndorff-Nielsen, Benth and Veraart~\cite{BNBV-first}. Note that we consider the ambit field $Z$ without drift and restrict 
our attention to times $t$ which are positive. Moreover, in the general definition of ambit fields 
by Barndorff-Nielsen and Schmiegel~\cite{BN-Schm}, the ambit set $A$ is also allowed to be dependent on
time and space $(t,x)$. We refrain from such generality here, as in most cases such dependency can be
included in the specification of the kernel function $g$. 

The objective of this paper is to define a class of Volterra processes with values in Hilbert space which 
provides an infinite-dimensional formulation of ambit fields. We shall call these processes
{\it Hambit fields}, referring to the Hilbert space-valued structure. After defining Hambit fields, we
discuss some specific examples and relate the Hambit fields to the "classical" ambit fields $Z(t,x)$ as
in \eqref{def-ambit-classic}. Under mild conditions, we can compute a rather explicit expression for the 
characteristic functional of a Hambit field. If $L$ is a Wiener basis, then the Hambit field becomes
a conditional Gaussian Hilbert-valued random variable. 

One of our main results is the representation of Hambit fields as a weighted series of volatility 
modulated Volterra processes. Volatility modulated Volterra processes generalize L\'evy 
semistationary  processes, for which Ornstein-Uhlenbeck processes constitute a particular case. 
L\'evy semistationary processes have been applied to model energy spot prices (see Barndorff-Nielsen, Benth and Veraart~\cite{BNBV-spot}), while in Barndorff-Nielsen, Benth and 
Veraart~\cite{BNBV-forward} ambit fields have been proposed as a model for 
energy forward markets. Thus, the representation of Hambit fields in terms
of a weighted series of volatility modulated Volterra processes provides us with a
useful theoretical link between spot and forward market models based on ambit fields. This result shows
the power of lifting ambit fields to Hilbert space, which gives a simple approach to show 
such a representation using basis function expansions. For an extensive discussion of energy
spot and forward markets and multi-factor commodity pricing models, we refer to  
Benth, \v{S}altyt\.{e} Benth and Koekebakker\cite{BSBK-book}.   

Hambit fields can be seen as a Volterra process in Hilbert space. By a simple splitting of
time in the integration and in the kernel function they can be viewed as mild solutions of a 
first order stochastic partial differential equation formulated in a Hilbert space of functions from
$\R_+$ into the state space of the Hambit field.  We construct an explicit space of such functions
on $\R_+$, generalizing the Filipovic space of real-valued absolutely continuous functions on
$\R_+$ (see Filipovic~\cite{F}). Via an evaluation map, we can transform the solution of the
stochastic partial differential equation linearly into a Hambit field. This result follows from a
commutativity property of the stochastic integral with linear maps. 

Using the interpretation of Hambit fields as the boundary solution of a stochastic partial
differential equations, we develop an iterative finite difference scheme. The scheme is 
formulated in the state space of the Hambit field, and under certain Lipschitz conditions
on the kernel function the convergence rate of the scheme is controlled. Our results provide a framework for
numerical studies of ambit fields, taking a different route than the Fourier-based method 
suggested by Eyjolfsson~\cite{E}. 

Our results are presented as follows. In the next section we define Hambit fields and study some elementary aspects and develop a series representation in terms of volatility modulated
Volterra processes. We proceed in Section~3 by introducing a stochastic partial
differential equation for which we can relate Hambit fields as a boundary solution. Finally,  
Section~4 is devoted to the development and analysis of a finite difference scheme for this
stochastic partial differential equation.

\section{Definition and analysis of Hambit fields}
In this Section we introduce a class Hilbert-space valued Volterra processes that provides a general definition of ambit fields as defined in \eqref{def-ambit-classic}.  

In the sequel, we shall operate with the three separable Hilbert spaces $\U, \V$ and $\Hi$, where we denote the respective inner products
by  $(\cdot,\cdot)_i$ and corresponding norms $|\cdot|_i, i=\U,\V,\Hi$. 
Let $t\mapsto\sigma(t)$ be a $\U$-valued predictable stochastic process. Introduce the measurable function 
$\Gamma:\R_+^2\rightarrow\mathcal{L}(\U,\mathcal{L}(\V,\Hi))$, where 
$\mathcal{L}(\V,\Hi)$ is the space
of bounded operators from $\V$ to $\Hi$, and $\mathcal{L}(\U,\mathcal{L}(\V,\Hi))$ the space of 
bounded operators from $\U$ to 
$\mathcal{L}(\V,\Hi)$. Note that since $\Hi$ is a Hilbert space, $\mathcal{L}(\V,\Hi)$ becomes a Banach space, 
which again implies that 
$\mathcal{L}(\U,\mathcal{L}(\V,\Hi))$ is a Banach space under respective operator norms.  By the predictability of the process $\sigma$, we find that  
$s\in[0,t]\mapsto\Gamma(s,t)(\sigma(s))\in\mathcal{L}(\V,\Hi)$ is predictable.
Finally, assume that $L$ is a square-integrable $\V$-valued L\'evy process with zero mean 
(i.e., $L$ is a martingale). Denote by $\mathcal{Q}\in\mathcal{L}(\V)$ the covariance operator of $L$, being a symmetric, non-negative definite trace
class operator. Note that we use the notation $\mathcal{L}(\V)$ for $\mathcal{L}(\V,\V)$, and that we do not assume 
independence between  $\sigma$ and $L$.

We define a {\it Hambit field} as 
follows:
\begin{definition}
Suppose, for each $t\leq T$,
\begin{equation}
\label{int-cond}
\E[\int_0^t\|\Gamma(t,s)(\sigma(s))\mathcal{Q}^{1/2}\|_{\text{HS}}^2\,ds]<\infty\,,
\end{equation}
where $\|\cdot\|_{\text{HS}}$ denotes the Hilbert-Schmidt norm on $\mathcal{L}(\V,\Hi)$.
Then the $\Hi$-valued stochastic process $\{X(t)\}_{t\in[0,T]}$ defined as
$$
X(t)=\int_0^t\Gamma(t,s)(\sigma(s))\,dL(s)\,,
$$ 
is called a {\it Hambit field}.
\end{definition}
We remark that by Peszat and Zabczyk~\cite[Sect.~8.6]{PZ}, the conditions on $\Gamma$ and $\sigma$ make the stochastic integral with respect to $L$ well-defined, in fact the following isometry holds
\begin{equation}
\label{isometry}
\E\left[\left| X(t) \right|_{\mathcal{H}}^2 \right] = \E\left[\int_0^t\|\Gamma(t,s)(\sigma(s))\mathcal{Q}^{1/2}\|_{\text{HS}}^2\,ds \right]\,.
\end{equation}
A convenient sufficient condition for \eqref{int-cond}
is formulated in the next Lemma:
\begin{lemma}
\label{lemma:suff-int-cond}
Suppose for each $t\leq T$ that 
$$
\int_0^t\|\Gamma(t,s)\|_{\text{op}}^2\E[|\sigma(s)|_{\U}^2]\,ds<\infty\,,
$$
then Condition~\eqref{int-cond} holds. Here, $\|\cdot\|_{\text{op}}$ denotes the operator norm in $\mathcal{L}(\U,\mathcal{L}(\V,\Hi))$.
\end{lemma} 
\begin{proof}
If $\{v_m\}_{m\in\mathbb{N}}$ is an ONB in $\V$, then by definition of the Hilbert-Schmidt norm and 
$\Gamma(t,s)(\sigma(s))\in L(\V,\Hi)$ yield
\begin{align*}
\|\Gamma(t,s)(\sigma(s))\mathcal{Q}^{1/2}\|_{\text{HS}}^2&=\sum_{m=1}^{\infty}|\Gamma(t,s)(\sigma(s))\mathcal{Q}^{1/2}v_m|_{\V}^2\\
&\leq \|\Gamma(t,s)(\sigma(s))\|^2_{\text{op}}\sum_{m=1}^{\infty}|\mathcal{Q}^{1/2}v_m|^2_{\V} \\
&\leq\|\Gamma(t,s)\|_{\text{op}}^2|\sigma(s)|^2_{\U}\|\mathcal{Q}^{1/2}\|_{\text{HS}}^2\,.
\end{align*} 
Since $Q$ is trace class operator, the result follows.
\end{proof}
We note that this sufficient condition on the integrability of the  "kernel function" $\Gamma$ and the "volatility" $\sigma$ share some similarity with the analogous condition for classical ambit fields (see \eqref{int-cond-ambit-classical}). 

Let us look an example of a Hambit field motivated by the analysis of Benth, R\"udiger and
S\"uss~\cite{BRS}. Consider a stochastic volatility modulated
Ornstein-Uhlenbeck process of the following form:
\begin{equation}
\label{def:SPDE}
dX(t)=\mathcal{A}X(t)\,dt+\sigma(t)\,dW(t)\,,\qquad X(0)=X_0\in\Hi\,,
\end{equation}
where $\mathcal{A}$ is a (possibly unbounded) linear operator on $\Hi$ which is densely defined and generating a
$C_0$-semigroup $\mathcal{S}$. Moreover, it is assumed that $W$ is an $\Hi$-valued Wiener process with covariance
operator $\mathcal{Q}$. Hence, we choose $\V=\Hi$. The volatility process $\sigma(t)$ is assumed to be predictable and
take values in the
space of Hilbert-Schmidt operators on $\Hi$, denoted $\mathcal{L}_{\text{HS}}(\Hi)$. Thus, we let 
$\U=\mathcal{L}_{\text{HS}}(\Hi)$, and recall 
that whenever $\Hi$ is a separable Hilbert space, $\mathcal{L}_{\text{HS}}(\Hi)$ becomes a separable Hilbert space
under the Hilbert-Schmidt norm. A mild solution of \eqref{def:SPDE} is
\begin{equation}
\label{def:SPDE:mild}
X(t)=\mathcal{S}_{t}X_0+\int_0^t\mathcal{S}_{t-s}\sigma(s)\,dW(s)\,.
\end{equation}
Note that the stochastic integral is well-defined as long as we have
\begin{equation}
\label{int-cond-sv}
\E\left[\int_0^t\|\mathcal{S}_{t-s}\sigma(s)\mathcal{Q}^{1/2}\|_{\text{HS}}^2\,ds\right]<\infty\,.
\end{equation}
Now, define $\Gamma(t,s)\in\mathcal{L}(\mathcal{L}_{\text{HS}}(\Hi))$ as
$\Gamma(t,s):\sigma\mapsto \mathcal{S}_{t-s}\sigma$. For any $\sigma\in\mathcal{L}_{\text{HS}}(\Hi)$, $\mathcal{S}_{t-s}\sigma$ becomes a linear bounded operator on $\Hi$, and since $\sigma$ is Hilbert-Schmidt, it follows that $\mathcal{S}_{t-s}\sigma$ is Hilbert-Schmidt as well. Hence, $\Gamma(t,s)$ 
maps linearly the Hilbert-Schmidt operators on $\Hi$ into itself. Moreover,  since we have
$$
\|\Gamma(t,s)\|_{\text{op}}=\sup_{\|\sigma\|_{\text{HS}}\leq 1}\|\Gamma(t,s)(\sigma)\|_{\text{HS}}=\sup_{\|\sigma\|_{\text{HS}}\leq 1}\|\mathcal{S}_{t-s}\sigma\|_{\text{HS}}\leq\|\mathcal{S}_{t-s}\|_{\text{op}}\sup_{\|\sigma\|_{\text{HS}}\leq 1}\|\sigma\|_{\text{HS}}
$$
and therefore  $\|\Gamma(t,s)\|_{\text{op}}\leq\|\mathcal{S}_{t-s}\|_{\text{op}}<\infty$.
By the general exponential growth bound on a $C_0$-semigroup and norm estimates on the 
Hilbert-Schmidt norm, we find
$$
\|\Gamma(t,s)(\sigma(s))\mathcal{Q}^{1/2}\|_{\text{HS}}^2=\|\mathcal{S}_{t-s}\sigma(s)\mathcal{Q}^{1/2}\|_{\text{HS}}^2\leq\|Q^{1/2}\|_{\text{op}}^2M\e^{w(t-s)}\|\sigma(s)\|_{\text{HS}}^2
$$
for positive constants $M$ and $w$. But then, according to Lemma~\ref{lemma:suff-int-cond}, 
it is sufficient that 
$$
\int_0^t\E[\|\sigma(s)\|_{\text{HS}}^2]\,ds<\infty\,,
$$
to ensure integrability. Thus, we conclude that the stochastic integral in $X$ defined in 
\eqref{def:SPDE:mild} is a Hambit field. 

In Benth, R\"udiger and S\"uss~\cite{BRS} a particular definition of the stochastic volatility process $\sigma$ is considered. Indeed, they propose a 
generalization of the BNS stochastic volatility model (see Barndorff-Nielsen and Shephard~\cite{BNS}) to operator-valued Ornstein-Uhlenbeck (OU) processes. To this end,
let $\mathcal{Y}(t)$ be a symmetric non-negative definite process with values in $\mathcal{L}_{\text{HS}}(\Hi)$ defined by the dynamics
$$
d\mathcal{Y}(t)=\mathbb{C}\mathcal{Y}(t)\,dt+d\mathcal{Z}(t)
$$
where $\mathcal{Z}(t)$ is an $\mathcal{L}_{\text{HS}}(\Hi)$-valued square integrable L\'evy process and $\mathbb{C}\in \mathcal{L}(\mathcal{L}_{\text{HS}}(\Hi))$. 
Under suitable conditions on $\mathbb{C}$ and $\mathcal{Z}$ we can ensure that $\mathcal{Y}(t)$ is a
symmetric, non-negative definite Hilbert-Schmidt operator (see Benth, R\"udiger and S\"uss~\cite{BRS}
for details). Moreover, following the arguments in Prop.  3.1
of Benth, R\"udiger and S\"uss~\cite{BRS}, we can show that 
$$
\text{Tr}(\mathcal{Y}(t))=\text{Tr}(\e^{\mathbb{C}t}\mathcal{Y}_0)+\text{Tr}(\int_0^t\e^{\mathbb{C}s}\,ds\E[\mathcal{Z}(1)])\,,
$$
and
$$
\E[\|\sigma(t)\|_{\text{HS}}^2]=\sum_{k=1}^{\infty}(\sigma^2(t)h_k,h_k)_{\Hi}=\text{Tr}(\mathcal{Y}(t))\,.
$$
Thus, as the $C_0$-semigroup $\exp(\mathbb{C}t)$ of $\mathbb{C}$ is Bochner integrable since $\mathbb{C}$ is bounded,
and $\mathcal{Z}(1)$ has finite expected value, it follows from the continuity of the Bochner integral that 
$t\mapsto\text{Tr}(\mathcal{Y}(t))$ is integrable on finite time intervals. This shows that we can use $\mathcal{Y}^{1/2}(t)$ as a
stochastic volatility process $\sigma$ in the definition of a Hambit field.

Let us return back to the general discussion of Hambit fields. Our next result concerns the $L^2$-proximity of two distinct Hambit fields.
\begin{lemma}
\label{lemma:approx-condition}
Suppose the Hambit fields
$$
X_i(t) = \int_0^t \Gamma_i(t,s)(\sigma_i(s)) \, dL(s)\,,\quad i=1,2\,,
$$
fulfill the premise of Lemma \ref{lemma:suff-int-cond}. Then,
\begin{align*}
\E\left[ \left| X_1(t) - X_2(t) \right|_\mathcal{H}^2 \right] &\leq \|\mathcal{Q}^{1/2}\|_{\text{HS}}^2\int_0^t\|\Gamma_1(t,s) - \Gamma_2(t,s)\|_{\text{op}}^2 \E\left[ \left| \sigma_1(s) \right|_\mathcal{U}^2 \right]\,ds\\
&\qquad+\|\mathcal{Q}^{1/2}\|_{\text{HS}}^2\int_0^t \| \Gamma_2(t,s) \|_{\text{op}}^2 \E \left[ |\sigma_1(s) - \sigma_2(s) |_{\mathcal{U}}^2 \right]\, ds\,,
\end{align*}
for all $t\geq 0$.
\end{lemma}
\begin{proof}
By the identity 
$$
\Gamma_1(t,s)(\sigma_1(s)) - \Gamma_2(t,s)(\sigma_2(s)) =  (\Gamma_1(t,s) - \Gamma_2(t,s))(\sigma_1(s))  + \Gamma_2(t,s)(\sigma_1(s) - \sigma_2(s)),
$$
the isomety \eqref{isometry} and (the proof of) Lemma \ref{lemma:suff-int-cond}, it holds that 
\begin{align*}
\E\left[ \left| \int_0^t (\Gamma_1(t,s) - \Gamma_2(t,s))(\sigma_1(s))dL(s) \right|_\mathcal{H}^2 \right] 
&\leq \|\mathcal{Q}^{1/2}\|_{\text{HS}}^2 \int_0^t \| \Gamma_1(t,s) - \Gamma_2(t,s) \|_{\text{op}}^2 \E\left[ \left| \sigma_1(s) \right|_\mathcal{U}^2 \right] ds\,,
\end{align*}
and 
\begin{align*}
\E\left[ \left| \int_0^t \Gamma_2(t,s)(\sigma_1(s) - \sigma_2(s)) dL(s) \right|_\mathcal{H}^2 \right] 
&\leq \|\mathcal{Q}^{1/2}\|_{\text{HS}}^2\int_0^t \| \Gamma_2(t,s) \|_{\text{op}}^2  \E\left[ \left| \sigma_1(s) -\sigma_2(s) \right|_\mathcal{U}^2 \right] ds\,, 
\end{align*}
from which the conclusion follows.
\end{proof}
As an application of the above result, we consider approximating a given Hambit field as follows: 
let $\Pi_n := \{s_i\}_{i=1}^n, n=1,2,\ldots,$ be a sequence of partitions of $[0,t]$ such that
$\max_{1 \leq i \leq n-1} |s_{i+1} - s_i| \downarrow 0$. Let, for each $n\in\mathbb{N}$,
\begin{equation}
\label{Hambit:Approx}
\Gamma_n(t,s) := \sum_{i=1}^{n-1} \Gamma(t,s_i) 1_{(s_i,s_{i+1}]}(s), \text{ and } \sigma_n(s) := \sum_{i=1}^{n-1} \sigma(s_i) 1_{(s_i,s_{i+1}]}(s),
\end{equation}
be the corresponding piecewise constant approximations of $\Gamma(t,\cdot)$ and $\sigma(\cdot)$ on $[0,t]$. Then, it follows by Lemma \ref{lemma:approx-condition} that the Hambit fields
\begin{equation}
\label{def:Xn}
X_n(t) := \int_0^t \Gamma_n(t,s)(\sigma_n(s)) \, dL(s)
\end{equation}
for $n\in\mathbb{N}$ approximate the original Hambit field, $X(t)$, if
\begin{equation}
\label{Hambit:RiemannSummability}
\int_0^t \left( \|\Gamma(t,s) - \Gamma_n(t,s)\|_{\text{op}}^2 \E\left[ \left| \sigma(s) \right|_\mathcal{U}^2 \right]+\| \Gamma_n(t,s) \|_{\text{op}}^2 \E \left[ |\sigma(s) - \sigma_n(s) |_{\mathcal{U}}^2 \right] \right) ds \to 0\,,
\end{equation}
when $n\rightarrow\infty$. Note that the approximative Hambit fields given by \eqref{def:Xn} are well defined if \eqref{Hambit:RiemannSummability} holds. Indeed, it follows by \eqref{Hambit:RiemannSummability} and Lemma \ref{lemma:approx-condition} that $\E[|X(t)-X_n(t)|_\Hi^2] \to 0$ when $n\rightarrow\infty$, which in turn means that 
$\lim_{n\rightarrow\infty}\E[|X_n(t)|_\Hi^2]=\E[|X(t)|_\Hi^2]$. 
For future reference we state the above convergence condition in an assumption.
\begin{assumption}\label{Assumption:Convergence}
A Hambit field $X(t)$ can be piecewise constantly approximated if condition \eqref{Hambit:RiemannSummability} is fulfilled, where $\Gamma_n(t,s)$ and $\sigma_n(s)$ are defined by \eqref{Hambit:Approx} and the limit is obtained by taking finer and finer partitions. 
\end{assumption}

We remark that the purpose of the above assumption is to identify conditions under which 
\begin{equation}\label{SimpleApprox}
\E[|X(t) - X_n(t) |_\Hi] \to 0\,,
\end{equation}
as we consider finer and finer partitions. Recall that the stochastic integral defining a Hambit field is built by first defining it for simple functions, and then extending it via the isometric formula \eqref{isometry}, which means that the simple functions are dense in the space of integrable functions. If the integrand 
$s \mapsto \Gamma(t,s)(\sigma(s))$ is continuous function from $[0,t]$ into the space of bounded linear operators with norm defined by 
$\|\cdot \mathcal{Q}^{1/2}\|_{\text{HS}}$, then one can choose the simple functions 
as in \eqref{Hambit:Approx}, and \eqref{SimpleApprox} follows by the isometric formula
\eqref{isometry}. 

Suppose $s \mapsto \Gamma(t,s)$ is continuous with respect to $\|\cdot\|_{\text{op}}$ on
$s\in[0,t]$, and assume $\sup_{s \in [0,t]}\E[|\sigma(s)|_\U^2] < \infty$ and 
$$
\int_0^t \|\Gamma(t,s)\|_{\text{op}}^2 \E\left[|\sigma(s) - \sigma_n(s)|_\U^2\,ds\right]\rightarrow0\,,
$$
when $n\rightarrow 0$. Then, Assumption~\ref{Assumption:Convergence} holds. Indeed,
by the triangle inequality
\begin{align*}
\int_0^t\|\Gamma_n(t,s)\|^2_{\text{op}}&\E\left[|\sigma(s)-\sigma_n(s)|_{\U}^2\right]\,ds \\
&\leq 2\int_0^t\|\Gamma(t,s)-\Gamma_n(t,s)\|_{\text{op}}^2\E\left[|\sigma(s)-\sigma_n(s)|_{\U}^2\right]\,ds \\
&\qquad+2\int_0^t\|\Gamma(t,s)\|_{\text{op}}^2\E\left[|\sigma(s)-\sigma_n(s)|_{\U}^2\right]\,ds\,.
\end{align*}  
By assumption, the second term above converges to zero as $n\rightarrow 0$. Consider the first term:
Note that 
\begin{align*}
\sup_{s\in[0,t]}\E\left[|\sigma(s)-\sigma_n(s)|_{\U}^2\right]&\leq2\sup_{s\in[0,t]}\E[|\sigma(s)|_{\U}^2]
+2\sup_{s\in[0,t]}\E[|\sigma_n(s)|_{\U}^2] \\
&\leq 4\sup_{s\in[0,t]}\E[|\sigma(s)|_{\U}^2]\,,
\end{align*}
since 
$$
\sup_{s\in[0,t]}\E[|\sigma_n(s)|_{\U}^2]=\sup_{s_i\in\Pi_n}\E[|\sigma(s_i)|^2_{\U}]
\leq\sup_{s\in[0,t]}\E[|\sigma(s)|_{\U}^2]\,.
$$
Hence,
\begin{align*}
\int_0^t\|\Gamma(t,s)-\Gamma_n(t,s)\|^2_{\text{op}}&\E\left[|\sigma(s)-\sigma_n(s)|^2_{\U}\right]\,ds
\\ 
&\leq \sup_{s\in[0,t]}\|\Gamma(t,s)-\Gamma_n(t,s)\|^2_{\text{op}}\int_0^t\E\left[|\sigma(s)-\sigma_n(s)|_{\U}^2\right]\,ds \\
&\leq 4t\sup_{s\in[0,t]}\E[|\sigma(s)|_{\U}^2]\sup_{s\in[0,t]}\|\Gamma(t,s)-\Gamma_n(t,s)\|^2_{\text{op}}\,,
\end{align*}
which tends to zero when $n\rightarrow\infty$ by uniform continuity. In conclusion, for these particular
regularity conditions on $\Gamma$ and $\sigma$ we are ensured that 
Assumption~\ref{Assumption:Convergence} holds. This case is particularly relevant when 
$\Gamma(t,s)$ is equal to a $C_0$-semigroup, $\Gamma(t,s) = \mathcal{S}_{t-s}$.

In the next Proposition we present the characteristic functional of the Hambit field:
\begin{proposition}
\label{prop:Hambit_cumulant}
Suppose that Assumption~\ref{Assumption:Convergence} holds and assume that 
$\sigma$ is independent of $L$. Then, for $h\in\Hi$, we have
$$
\E\left[\exp\left(\mathrm{i}(h,X(t))_{\Hi}\right)\right]=\E\left[\exp\left(\int_0^t\Psi_{L}\left((\Gamma(t,s)(\sigma(s)))^*h\right)\,ds\right)\right]\,,
$$
where $\Psi_L$ is the cumulant functional of $L(1)$.
\end{proposition}
\begin{proof}
Let $\{s_i\}_{i=1}^n$ be a partition of $[0,t]$ and denote $\Delta s_i=s_{i+1}-s_i$ and 
$\Delta L(s_i)=L(s_{i+1})-L(s_i)$ for $i=1,\ldots,n-1$. Then, by the independent increment property of 
$L$ and double conditioning using the independence between $\sigma$ and $L$, we find 
\begin{align*}
&\E\left[\exp\left(\mathrm{i}(h,\sum_{i=1}^{n-1}\Gamma(t,s_i)(\sigma(s_i))\Delta L(s_i))_{\Hi}\right)\right] \\
&\qquad\qquad\qquad=\E\left[\E\left[\exp\left(\mathrm{i}(h,\sum_{i=1}^{n-1}\Gamma(t,s_i)(\sigma(s_i))\Delta L(s_i))_{\Hi}\right)\,\vert\,\sigma(\cdot)\right]\right] \\
&\qquad\qquad\qquad=\E\left[\prod_{i=1}^{n-1}\E\left[\exp\left(\mathrm{i}(h,\Gamma(t,s_i)(\sigma(s_i))\Delta L(s_i))_{\Hi}\right)\,\vert\,\sigma(\cdot)\right]\right] \\
&\qquad\qquad\qquad=\E\left[\prod_{i=1}^{n-1}\E\left[\exp\left(\mathrm{i}((\Gamma(t,s_i)(\sigma(s_i)))^*h,\Delta L(s_i))_{\Hi}\right)\,\vert\,\sigma(\cdot)\right]\right] \\
&\qquad\qquad\qquad=\E\left[\prod_{i=1}^{n-1}\exp\left(\Psi_L((\Gamma(t,s_i)(\sigma(s_i)))^*h)\Delta s_i\right)\right]\,.
\end{align*}
The last equality follows from the L\'evy-Kintchine formula for $L$ 
(see Peszat and Zabczyk~\cite[Thm.~4.27]{PZ}). By the Cauchy-Schwarz inequality it holds that
\begin{align*}
\E\left[\left|(h,X(t))_{\Hi} - (h,X_n(t))_{\Hi} \right| \right] 
\leq |h|_\Hi\E\left[|X(t) - X_n(t)|^2_\Hi\right]^{1/2}\,,
\end{align*}
where $X_n(t)$ is defined by \eqref{def:Xn}. Thus, invoking the inequality $|\e^{\mathrm{i}x} - \e^{\mathrm{i}y}| \leq |x - y|$, for $x,y \in \R$, and Lemma~\ref{lemma:approx-condition}, 
complete the proof.
\end{proof}
Consider $L=W$, a Wiener process in $\Hi$. Then, the cumulant functional of $W(1)$ is $\Psi_W(v)=-\frac12(\mathcal{Q}v,v)_{\V}$ (see Peszat and Zabczyk~\cite[Thm.~4.27]{PZ}). 
If $\sigma$ is independent of $W$, we find by Proposition~\ref{prop:Hambit_cumulant} 
that for any $h\in\Hi$
\begin{align*}
\E\left[\exp\left(\mathrm{i}(h,X(t))_{\Hi}\right)\right] &=\E\left[\exp\left(-\frac12\int_0^t(\mathcal{Q}(\Gamma(t,s)(\sigma(s)))^*h,(\Gamma(t,s)(\sigma(s)))^*h)_{\V}\,ds\right)\right] \\
&=\E\left[\exp\left(-\frac12\int_0^t(h,\Gamma(t,s)(\sigma(s))\mathcal{Q}(\Gamma(t,s)(\sigma(s)))^*h)_{\Hi}\,ds\right)\right] \\
&=\E\left[\exp\left(-\frac12(h,\int_0^t\Gamma(t,s)(\sigma(s))\mathcal{Q}(\Gamma(t,s)(\sigma(s)))^*\,ds h)_{\Hi}\right)\right] \,.
\end{align*}
We interpret the $ds$-integral in the last expectation as a Bochner integral in the space of operators. In conclusion, for $L=W$ and $\sigma$ independent of $W$, the Hambit field becomes
a Gaussian random variable conditional on $\sigma$.  Indeed, $X(t)\vert_{\sigma(\cdot)}$ is an
$\Hi$-valued Gaussian process with covariance operator
$$
\mathcal{Q}_{X(t)\vert_{\sigma(\cdot)}}=\int_0^t\Gamma(t,s)(\sigma(s))\mathcal{Q}(\Gamma(t,s)(\sigma(s)))^*\,ds
$$
and mean equal to zero. 


We discuss stationarity for the Hambit process. 
Let $\Gamma(t,s):=\Gamma(t-s)$ for a moment. Choosing a non-random time-independent volatility $\sigma(s):=\sigma\in\V$, we obtain the characteristic functional of the form 
$$
\E\left[\exp\left(\mathrm{i}(h,X(t))_{\Hi}\right)\right]=\exp\left(\int_0^t\Psi_{L}\left((G(s))^*h\right)\,ds\right)\,,
$$
where $G(s):=\Gamma(s)(\sigma)$. If $s\mapsto\Psi_{L}\left((G(s)^*h\right)\in L^1(\R_+)$, then we see that the characteristic functional of $X(t)$ has a limit
$$
\lim_{t\rightarrow\infty}\E\left[\exp\left(\mathrm{i}(h,X(t))_{\Hi}\right)\right]=\exp\left(\int_0^{\infty}\Psi_{L}\left((G(s))^*h\right)\,ds\right)\,.
$$
Assuming $\int_0^{\infty}\|G(s)\|_{\text{op}}^2\,ds<\infty$, we can define the $\Hi$-valued process
\begin{equation}
X_{\text{stat}}(t)=\int_{-\infty}^tG(s)\,dL(s)\,,
\end{equation}
which has characteristic functional
$$
\E\left[\exp\left(\mathrm{i}(h,X_{\text{stat}}(t))_{\Hi}\right)\right]=\exp\left(\int_0^{\infty}\Psi_{L}\left((G(s))^*h\right)\,ds\right)\,.
$$
Hence, $X(t)$, when $t\rightarrow\infty$ is equal in distribution to $X_{\text{stat}}(t)$. The process
$X_{\text{stat}}(t)$ is the stationary version of $X(t)$. We remark that ambit fields are often
defined to be stationary processes (see Barndorff-Nielsen and Schmiegel~\cite{BN-Schm}).
Letting $L=W$ again, we find that the stationary distribution of $X$ is Gaussian in $\Hi$ with covariance operator 
$$
\mathcal{Q}_{X_{\text{stat}}}=\int_0^{\infty}G(s)\mathcal{Q}G(s)^*\,ds\,,
$$
and mean equal to zero. As a specific example of an Hambit field which is asymptotically stationary, we might consider the Ornstein-Uhlenbeck process \eqref{def:SPDE} with constant non-random
volatility. In this case $\Gamma(t,s) = \mathcal{S}_{t-s}$, where $\mathcal{S}$ is the $C_0$-semigroup generated by $\mathcal{A}$.

\subsection{Relation to classical ambit fields}
We relate Hambit fields to the classical definition of ambit fields, see \eqref{def-ambit-classic}.
 
Let $\U$ be a Hilbert space of real-valued functions on a Borel measurable subset $A\subset\R^n$, $n\in\N$. 
Consider the measurable real-valued function $(t,s,x,y)\mapsto g(t,s,x,y)$, where $0\leq s\leq t<\infty$, $x\in B$, $y\in A$, and $B\subset\R^d, d\in\N$ being a Borel measurable subset. Let $\V$ be a Hilbert space of measures on the Borel subsets of $A$. For $\sigma\in\U$, we define the linear operator on $\V$
$$
\Gamma(t,s)(\sigma):=\int_{A}g(t,s,\cdot,y)\sigma(y)\,,
$$
given by 
$$
\Gamma(t,s)(\sigma)\mu=\int_{A}g(t,s,\cdot,y)\sigma(y)\,\mu(dy)\,,
$$
for any $\mu\in\V$. If we let $\Hi$ be a Hilbert space of real-valued functions on $B$, then under appropriate hypotheses on $g$ and
selection of Hilbert spaces one can have $\Gamma(t,s)(\sigma)\mu\in\Hi$ for $\mu\in\V$ and $\Gamma(t,s)\in L(\U,L(\V,\Hi))$. 
Assume $\sigma(s)$ is a $\U$-valued stochastic process such that $\Gamma(t,s)(\sigma(s))$ is integrable with respect to the 
$\V$-valued L\'evy process $L$. Then we get,
$$
X(t,x)=\int_0^t\int_{A}g(t,s,x,y)\sigma(s,y)\,L(dy,ds)\,,
$$
which is a classical ambit field. Note that we choose here to work with a kernel
function $g$ which is non-stationary in time. In $X(t,x)$ above, the L\'evy process $L$ is a measure.  A L\'evy basis is not
a measure, but very close to one (see
Barndorff-Nielsen, Benth and Veraart~\cite{BNBV-first} for a discussion of Hilbert-valued processes and 
L\'evy bases).  

To be more specific, choose $n=d=1$ and let $A=B=\R_+$. Assume $\U=\V=\Hi$, and let $\U$ be the Filipovic space (see Filipovic~\cite{F})
of absolutely continuous functions on $\R_+$, that is, real-valued functions $f$ on $\R_+$ which are weakly differentiable and such that
\begin{equation}
|f|_w^2:=f^2(0)+\int_0^{\infty}w(y)|f'(y)|^2\,dy\,,
\end{equation}  
for a non-decreasing weight function $w:\R_+\rightarrow[1,\infty)$ satisfying $\int_0^{\infty}w^{-1}(y)\,dy<\infty$. We denote 
this separable Hilbert space $\U_w$, and its inner product by $(\cdot,\cdot)_w$. For $\sigma\in\U_w$, we need to impose conditions on
$g$ such that
$$
\Gamma(t,s)(\sigma)f=\int_0^{\infty}g(t,s,\cdot,y)\sigma(y)f'(y)\,dy
$$ 
is an element in $\U_w$ for all $(t,s)$ with $s\leq t<\infty$ and $f\in\U_w$.
Next, we need to have that $\Gamma(t,s)(\sigma)\in L(\U_w)$ and $\Gamma(t,s)\in L(\U_w,L(\U_w))$,
and furthermore that $s\mapsto\Gamma(t,s)(\sigma(s))$ is integrable with respect to the 
$\U_w$-valued L\'evy process $L$. We collect the conditions in the next Lemma:
\begin{lemma}
\label{lemma:classic-ambit-filipovic}
Let $\sigma$ be a predictable $\U_w$-valued process, and suppose that $x\mapsto g(t,s,x,y)\in\U_w$ for $a.e.$ $(t,s,y)$ is such that
\begin{itemize}
\item for $a.e.$ $t\geq s\geq 0$
$$
\int_0^{\infty}w^{-1}(y)|g(t,s,\cdot,y)|^2_w\,dy<\infty\,,
$$
\item and $t\geq 0$,
$$
\int_0^{\infty}w^{-1}(y)\int_0^t|g(t,s,\cdot,y)|^2_w\E[|\sigma(s)|^2_w]\,ds\,dy<\infty\,.
$$
\end{itemize}
Then we have a classical ambit field
$$
X(t,x)=\int_0^t\int_0^{\infty}g(t,s,x,y)\sigma(y)L(dy,ds)
$$
with $X(t,\cdot)\in\U_w$ for $t<\infty$.
\end{lemma}
\begin{proof}
For $\sigma_1,\sigma_2\in\U_w$, we obviously have $\Gamma(t,s)(\sigma_1+\sigma_2)=\Gamma(t,s)(\sigma_1)+\Gamma(t,s)(\sigma_2)$. Moreover, 
for $f_1,f_2\in\U_w$, it is also straightforward to see that $\Gamma(t,s)(\sigma)(f_1+f_2)=\Gamma(t,s)(\sigma)f_1+\Gamma(t,s)(\sigma)f_2$. Thus, to prove
that $\Gamma(t,s)\in L(\U_w,L(\U_w))$ we must show that the linear operators are bounded.

To this end, note that
$$
\|\Gamma(t,s)\|_{\text{op}}=\sup_{|\sigma|_w\leq 1}\|\Gamma(t,s)(\sigma)\|_{\text{op}}=\sup_{|\sigma|_w,|f|_w\leq 1}|\int_0^{\infty}g(t,s,\cdot,y)\sigma(y)f'(y)\,dy|_w\,.
$$
By definition
\begin{equation}
\label{norm-estimate}
|\Gamma(t,s)(\sigma)f|^2_w=(\int_0^{\infty}g(t,s,0,y)\sigma(y)f'(y)\,dy)^2+\int_0^{\infty}w(x)(\int_0^{\infty}g_x(t,s,x,y)\sigma(y)f'(y)\,dy)^2\,dy\,,
\end{equation}
where $g_x$ denotes the weak derivative with respect to the third argument of $g$. By the Cauchy-Schwartz inequality, we find for the first term
\begin{align*}
(\int_0^{\infty}g(t,s,0,y)\sigma(y)f'(y)\,dy)^2&=(\int_0^{\infty}w^{-1/2}(y)g(t,s,0,y)\sigma(y)w^{1/2}(y)f'(y)\,dy)^2 \\
&\leq\int_0^{\infty}w^{-1}(y)g^2(t,s,0,y)\sigma^2(y)\,dy\int_0^{\infty}w(y)|f'(y)|^2\,dy \\
&\leq \int_0^{\infty}w^{-1}(y)g^2(t,s,0,y)\sigma^2(y)\,dy|f|_w^2\,.
\end{align*}
But, by the fundamental theorem of calculus and Cauchy-Schwartz' inequality again,
\begin{align*}
\sigma^2(y)&=(\sigma(0)+\int_0^y\sigma'(z)\,dz)^2 \\ 
&\leq 2\sigma^2(0)+2(\int_0^y\sigma'(z)\,dz)^2 \\
&=2\sigma^2(0)+2(\int_0^yw^{-1/2}(z)w^{1/2}(z)\sigma'(z)\,dz)^2 \\
&\leq 2\sigma^2(0)+2\int_0^yw^{-1}(z)\,dz\int_0^yw(z)|\sigma'(z)|^2\,dz \\
&\leq 2(1+\int_0^{\infty}w^{-1}(z)\,dz)|\sigma|_w^2\,.
\end{align*}
Hence, we find
$$
(\int_0^{\infty}g(t,s,0,y)\sigma(y)f'(y)\,dy)^2\leq 2\int_0^{\infty}w^{-1}(y)g^2(t,s,0,y)\,dy(1+\int_0^{\infty}w^{-1}(y)\,dy)|\sigma|_w^2|f|_w^2\,.
$$
For the second integral in \eqref{norm-estimate}, it follows by similar arguments that,
\begin{align*}
\int_0^{\infty}w(x)&(\int_0^{\infty}g_x(t,s,x,y)\sigma(y)f'(y)\,dy)^2\,dy \\
&=\int_0^{\infty}w(x)(\int_0^{\infty}w^{-1/2}(y)g_x(t,s,x,y)\sigma(y)w^{1/2}(y)f'(y)\,dy)^2\,dx \\
&\leq\int_0^{\infty}w(x)\int_0^{\infty}w^{-1}(y)g_x^2(t,s,x,y)\sigma^2(y)\,dy\,dx|f|_w^2 \\
&=\int_0^{\infty}\int_0^{\infty}w(x)w^{-1}(y)g^2_x(t,s,x,y)\,dx\sigma^2(y)\,dy|f|_w^2 \\
&\leq2\int_0^{\infty}w^{-1}(y)(\int_0^{\infty}w(x)g_x^2(t,s,x,y)\,dx)\,dy(1+\int_0^{\infty}w^{-1}(z)\,dz)|\sigma|_w^2|f|_w^2\,.
\end{align*}
These estimations imply 
$$
\|\Gamma(t,s)(\sigma)\|^2_{\text{op}}\leq |\sigma|^2_w 2(1+\int_0^{\infty}w^{-1}(z)\,dz)\int_0^{\infty}w^{-1}(y)|g(t,s,\cdot,y)|^2_w\,dy\,,
$$
and 
$$
\|\Gamma(t,s)\|^2_{\text{op}}\leq 2(1+\int_0^{\infty}w^{-1}(z)\,dz)\int_0^{\infty}w^{-1}(y)|g(t,s,\cdot,y)|^2_w\,dy\,,
$$
and therefore $\Gamma\in L(\U_w,L(\U_w))$ by the assumptions of the Lemma. 

For the $L$-integrability, we first note that since $\sigma(s)$ is assumed to be predictable, it
follows that $s\mapsto \Gamma(s,t)(\sigma(s))$ is predictable. We must show that the integrability condition \eqref{int-cond} holds:
\begin{align*}
\E\left[\int_0^t\|\Gamma(t,s)(\sigma(s))\mathcal{Q}^{1/2}\|_{\text{HS}}^2\,ds\right]&\leq\E\left[\int_0^t\|\Gamma(t,s)(\sigma(s))\|_{\text{op}}^2\|\mathcal{Q}^{1/2}\|_{\text{HS}}^2\,ds\right] \\
&\leq2\|\mathcal{Q}^{1/2}\|_{\text{HS}}^2(1+\int_0^{\infty}w^{-1}(z)\,dz) \\
&\qquad\qquad\times\int_0^{\infty}w^{-1}(y)\int_0^t|g(t,s,\cdot,y)|_w^2\E\left[|\sigma(s)|^2_w\right]\,ds\,dy \,,
\end{align*}
which is finite by the assumptions of the Lemma. Hence, the proof is complete.
\end{proof}
In the representation of $X(t,x)$, we have used the notation $L(dy,ds)=\partial_y L(y,ds)\,dy$, where $\partial_y$ is the partial (weak) derivative with respect
to $y$. We remark that we can define the classical ambit field
$$
X(t,x)=\int_0^t\int_{A}g(t,s,x,y)\sigma(s)\,L(dy,ds)
$$
by
$$
X(t,x)=\int_0^t\int_0^{\infty}g(t,s,x,y)\mathrm{1}(y\in A)\sigma(s)\,L(dy,ds)\,,
$$
for some Borel measurable subset $A\subset\R_+$. We note that 
$$
|g(t,s,\cdot,y)\mathrm{1}(y\in A)|_w=\mathrm{1}(y\in A)|g(t,s,\cdot,y)|_w\,,
$$
and therefore
$$
\int_0^{\infty}w^{-1}(y)|g(t,s,\cdot,y)\mathrm{1}(y\in A)|_w^2\,dy=\int_{A}w^{-1}(y)|g(t,s,\cdot,y)|_w^2\,dy\leq
\int_0^{\infty}w^{-1}(y)|g(t,s,\cdot,y)|_w^2\,dy\,.
$$
Thus, either we can impose slightly weaker $\U_w$-norm-integrability conditions on $g$ (the middle estimate above), or we can assume the strong one given in Lemma~\ref{lemma:classic-ambit-filipovic}.
In the latter case, we observe that such a condition provides us with a classical ambit field for {\it all} choices of $A$.  
%

\subsection{Representation of Hambit fields in terms of volatility modulated Volterra processes}

We show that a Hambit field can be represented as a countable sum of volatility modulated Volterra
(VMV) processes under a certain regularity condition
on the stochastic volatility field $\sigma$: 
\begin{proposition}
\label{prop:ambit-lss}
Let $\{u_n\}_{n\in\mathbb{N}}, \{v_m\}_{m\in\mathbb{N}}$ and $\{h_k\}_{k\in\mathbb{N}}$ be 
ONB's in $\U, \V$ and
$\Hi$, resp. Suppose that 
$$
\int_0^t\|\Gamma(t,s)\|^2_{\text{op}}\left(\sum_{n=1}^{\infty}\E[(\sigma(s),u_n)^2_{\U}]^{1/2}\right)^2\,ds<\infty\,.
$$
Then the Hambit field $X(t)$ can be represented in $L^2(\Omega)$ as
$$
X(t)=\sum_{n,m,k=1}^{\infty}Y_{n,m,k}(t)\,h_k\,,
$$
where $t\mapsto Y_{n,m,k}(t)$, $0\leq t\leq T, n,m,k\in\mathbb{N}$ are real-valued VMV processes defined by
$$
Y_{n,m,k}(t)=\int_0^t(\Gamma(t,s)(u_n)v_m,h_k)_{\Hi}(\sigma(s),u_n)_{\U}\,dL_m(s)\,,
$$
and $L_m:=(L,v_m)_{\V}, m\in\mathbb{N}$ are real-valued square integrable L\'evy processes with zero mean. 
\end{proposition}
\begin{proof}
We can represent $L(t)$ by
$$
L(t)=\sum_{m=1}^{\infty}(L(t),v_m)_{\V}v_m\,,
$$
where $L_m:=(L,v_m)_{\V}$ is a real-valued square integrable mean zero L\'evy process. Hence,
$$
X(t)=\sum_{m=1}^{\infty}\int_0^t\Gamma(t,s)(\sigma(s))v_m\,dL_m(s)\,.
$$
But $\Gamma(t,s)(\sigma(s))v_m\in\Hi$, and thus the stochastic integral $\int_0^t\Gamma(t,s)(\sigma(s))v_m\,dL_m(s)\in\Hi$
as well. Hence, $a.s.$,
\begin{align*}
\int_0^t\Gamma(t,s)(\sigma(s))v_m\,dL_m(s)&=\sum_{k=1}^{\infty}(\int_0^t\Gamma(t,s)(\sigma(s))v_m\,dL_m(s),h_k)_{\Hi}h_k \\
&=\sum_{k=1}^{\infty}\int_0^t(\Gamma(t,s)(\sigma(s))v_m,h_k)_{\Hi}\,dL_m(s) h_k\,.
\end{align*}
The last equality follows by definition of the stochastic integral of an $\Hi$-valued adapted process
with respect to a real-valued L\'evy process.
This means that
$$
X(t)=\sum_{m,k=1}^{\infty}\int_0^t(\Gamma(t,s)(\sigma(s))v_m,h_k)_{\Hi}\,dL_m(s) h_k\,.
$$
Finally, express $\sigma(s)=\sum_{n=1}^{\infty}(\sigma(s),u_n)_{\U}u_n$ to find
$$
(\Gamma(t,s)(\sigma(s))v_m,h_k)_{\Hi}=\sum_{n=1}^{\infty}(\sigma(s),u_n)_{\U}(\Gamma(t,s)(u_n)v_m,h_k)_{\Hi}\,,
$$
by linearity of the inner product and continuity of the operator $\Gamma(t,s)$. We show next that 
$$
\int_0^t(\Gamma(t,s)(\sigma(s))v_m,h_k)_{\Hi}\,dL_m(s)=\sum_{n=1}^{\infty}\int_0^t(\Gamma(t,s)(u_n)v_m,h_k)_{\Hi}(\sigma(s),u_n)_{\U}\,dL_m(s)\,.
$$
Note that, as the L\'evy process $L_m$ is a square-integrable martingale, we find by the definition of stochastic integration 
with respect to martingales (see e.g. Protter~\cite{Pr})
\begin{align*}
&\E\left[\left(\int_0^t(\Gamma(t,s)(\sigma(s))v_m,h_k)_{\Hi}\,dL_m(s)-\sum_{n=1}^{N}\int_0^t(\Gamma(t,s)(u_n)v_m,h_k)_{\Hi}(\sigma(s),u_n)_{\U}\,dL_m(s)\right)^2\right] \\
&\qquad\qquad=\E\left[\left(\int_0^t\sum_{n=N+1}^{\infty}(\Gamma(t,s)(u_n)v_m,h_k)_{\Hi}(\sigma(s),u_n)_{\U}\,dL_m(s)\right)^2\right] \\
&\qquad\qquad=\E[L^2_m(1)]\int_0^t\E\left[\left(\sum_{n=N+1}^{\infty}(\Gamma(t,s)(u_n)v_m,h_k)_{\Hi}(\sigma(s),u_n)_{\U}\right)^2\right]\,ds
\end{align*}
By Minkowski's inequality (see Folland~\cite[p.~186]{Folland}), we have
\begin{align*}
\E&\left[\left(\sum_{n=N+1}^{\infty}(\Gamma(t,s)(u_n)v_m,h_k)_{\Hi}(\sigma(s),u_n)_{\U}\right)^2\right]^{1/2} \\
&\qquad\qquad\leq
\sum_{n=N+1}^{\infty}\E\left[(\Gamma(t,s)(u_n)v_m,h_k)^2_{\Hi}(\sigma(s),u_n)^2_{\U}\right]^{1/2} \\
&\qquad\qquad\leq\|\Gamma(t,s)\|_{\text{op}}\sum_{n=N+1}^{\infty}\E\left[(\sigma(s),u_n)_{\U}^2\right]^{1/2}\,,
\end{align*}
since, using that the bases are orthonormal,
$$
|(\Gamma(t,s)(u_n)v_m,h_k)_{\Hi}|^2\leq |\Gamma(t,s)(u_n)v_m)|_{\Hi}^2|h_k|_{\Hi}^2\leq \|\Gamma(t,s)(u_n)\|_{\text{op}}^2
|v_m|_{\V}^2\leq\|\Gamma(t,s)\|_{\text{op}}^2|u_n|_{\U}^2\,.
$$
Hence, 
\begin{align*}
\int_0^t\E&\left[\left(\sum_{n=N+1}^{\infty}(\Gamma(t,s)(u_n)v_m,h_k)_{\Hi}(\sigma(s),u_n)_{\U}\right)^2\right]\,ds \\
&\qquad\qquad\leq
\int_0^t\|\Gamma(t,s)\|_{\text{op}}^2\left(\sum_{n=N+1}^{\infty}\E[(\sigma(s),u_n)_{\U}^2]^{1/2}\right)^2\,ds\,,
\end{align*}
which tends to zero as $N\rightarrow\infty$ by assumption. The result follows.

\end{proof}
Remark that the real-valued L\'evy processes $\{L_m\}_{m=1}^{\infty}$ defined in 
Prop.~\ref{prop:ambit-lss} above are not 
independent. They are not even zero correlated unless the ONB $\{v_m\}_{k\in\mathbb{N}}$ consists of the eigenvectors of 
$\mathcal{Q}$. Indeed, we have
$$
\E[(L(t),v_m)_{\V}(L(t),v_k)_{\V}]=(\mathcal{Q}v_m,v_k)_{\V}\,,
$$ 
for $k,m\in\N$. Further, we also observe that if $\Gamma(t,s)=\Gamma(t-s)$, i.e., the kernel is specified
in a stationary form, then the real-valued processes $Y_{n,m,k}(t)$ in Prop.~\ref{prop:ambit-lss}
become,
$$
Y_{n,m,k}(t)=\int_0^t(\Gamma(t-s)(u_n)v_m,h_k)_{\Hi}(\sigma(s),u_n)_{\U}\,dL_m(s)\,,
$$
which is in fact a L\'evy semistationary (LSS) process. Barndorff-Nielsen, Benth and 
Veraart~\cite{BNBV-spot} applied LSS processes to model spot prices in energy market. Further,
using factor models involving L\'evy-driven continuous-time autoregressive moving average processes
to describe electricity spot prices,
Benth et al.~\cite{BKMV} extended the classical commodity spot market models based on Wiener-driven Ornstein-Uhlenbeck processes. The class of continuous-time autoregressive moving average
processes is a special case of LSS processes (see Brockwell~\cite{Br}, and Benth and \v{S}altyt\.{e} Benth~\cite{BSB-book} for an analysis and discussion in weather modelling).
Barndorff-Nielsen, Benth and Veraart~\cite{BNBV-forward} proposed ambit fields as a modeling tool for energy forward markets. Our result in Prop.~\ref{prop:ambit-lss} shows that any ambit field can be represented as an infinite LSS (or VMV) factor model, providing a strong theoretical argument for the rationale in using LSS (or VMV) processes and ambit fields as modelling devices for
commodity market prices. 

The integrability condition on $\Gamma$ and $\sigma$ in Prop.~\ref{prop:ambit-lss} is stronger than the sufficient condition
in Lemma~\ref{lemma:suff-int-cond} for well-definedness of the Hambit field $X$. In fact, by Parseval's identity (and Tonelli's theorem)
$$
\E[|\sigma(s)|_{\U}^2]=\sum_{n=1}^{\infty}\E[(\sigma(s),u_n)_{\U}^2]
$$
for $\{u_n\}_{n\in\mathbb{N}}$ ONB of $\U$. As long as $\sum_{n=1}^{\infty}\E[(\sigma(s),u_n)_{\U}^2]^{1/2}<\infty$,
there exists $N\in\mathbb{N}$ such that $\E[(\sigma(s),u_n)_{\U}^2]^{1/2}<1$ for $n\geq N$. Hence
$\E[(\sigma(s),u_n)_{\U}^2]\leq\E[(\sigma(s),u_n)_{\U}^2]^{1/2}$. Thus, the condition in Prop.~\ref{prop:ambit-lss}
implies that the condition of Lemma~\ref{lemma:suff-int-cond} holds. Suppose now that $\{a_n\}_{n\in\mathbb{N}}$ is
a sequence of strictly positive numbers such that $\sum_{n=1}^{\infty}a_n^{-1}<\infty$. Then, by the
Cauchy-Schwarz inequality
$$
\left(\sum_{n=1}^{\infty}\E[(\sigma(s),u_n)^2_{\U}]^{1/2}\right)^2=\left(\sum_{n=1}^{\infty}a_n^{-1}\right)\left(\sum_{n=1}^{\infty}
a_n\E[(\sigma(s),u_n)^2_{\U}]\right)\,.
$$
Thus, 
\begin{equation}
\sum_{n=1}^{\infty}a_n\int_0^t\|\Gamma(t,s)\|_{\text{op}}^2\E[(\sigma(s),u_n)^2_{\U}]\,ds<\infty\,,
\end{equation}
is a sufficient condition for Prop.~\ref{prop:ambit-lss} to hold.

Let us consider an example. Let $U$ be a $\U$-valued square-integrable L\'evy process with zero mean and
$\eta\in L^2(\R_+)$. Assume that $\sigma(t)$ is the $\U$-valued OU process
$$
\sigma(t)=\int_0^t\eta(t-s)\,dU(s)\,.
$$
This is a very simple definition of an LSS-process with values in the Hilbert space $\U$. 
As $U$ is square-integrable, it has a covariance operator $\mathcal Q_U$ on $\U$, and we assume that the ONB $\{u_n\}$
is the set of eigenvectors of $\mathcal{Q}_U$ with corresponding eigenvalues $\lambda_n$. As $\mathcal{Q}_{U}$ is positive definite, we have $0\leq(\mathcal{Q}_{U}u_n,u_n)_{\U}=\lambda_n$, i.e., all eigenvalues are non-negative. We find that
$$
(\sigma(t),u_n)_{\U}=\sum_{k=1}^{\infty}\int_0^t(\eta(t-s)u_k,u_n)_{\U}\,dU_k(s)=\sum_{k=1}^{\infty}\int_0^t\eta(t-s)(u_k,u_n)_{\U}\,dU_k(s)=\int_0^t\eta(t-s)\,dU_n(s)
$$ 
where $U_n(s)=(U(s),u_n)_{\U}$ is a square integrable real-valued L\'evy process with zero mean.  But then,
$$
\E[(\sigma(t),u_n)^2_{\U}]=\E[(\int_0^t\eta(t-s)\,dU_n(s))^2]=\lambda_n\int_0^t\eta^2(s)\,ds\,.
$$
The integrability condition in Prop.~\ref{prop:ambit-lss} thus becomes
$$
\int_0^t\|\Gamma(t,s)\|_{\text{op}}^2\left(\sum_{n=1}^{\infty}\sqrt{\lambda_n}(\int_0^s\eta^2(v)\,dv)^{1/2}\right)^2\,ds=\left(\sum_{n=1}^{\infty}\sqrt{\lambda_n}\right)^2\int_0^t\|\Gamma(t,s)\|_{\text{op}}^2\int_0^s\eta^2(v)\,dv\,ds\,.
$$
Hence, since $\eta\in L^2(\R_+)$, the integrability condition in Prop.~\ref{prop:ambit-lss} is satisfied if 
$\int_0^t\|\Gamma(t,s)\|_{\text{op}}^2\,ds<\infty$ and $\sum_{n=1}^{\infty}\sqrt{\lambda_n}<\infty$. Note that $\infty>\|\mathcal{Q}_{\U}\|_{\text{HS}}^2=\sum_{n=1}^{\infty}\lambda_n$, which is weaker than the summability of $\sqrt{\lambda_n}$. We find that $\text{Tr}(\mathcal{Q}_{\U}^{1/2})=\sum_{n=1}^{\infty}\sqrt{\lambda_n}$, so the summability of $\sqrt{\lambda_n}$ is equivalent to assuming that $\mathcal{Q}_{\U}^{1/2}$ has finite trace. 

A natural application of Prop.~\ref{prop:ambit-lss} is to truncate the infinite sum in order
to obtain an approximation of the Hambit field $X$. 
For this purpose, define
\begin{equation}\label{def:XNMK}
X_{N,M,K}(t) := \sum_{n=1}^N \sum_{m=1}^M \sum_{k=1}^K Y_{n,m,k}(t) h_k,
\end{equation}
for $N,M,K \geq 1$, where $Y_{n,m,k}(t)$ is given in Proposition \ref{prop:ambit-lss}. 
It moreover follows by a repeated application of Minkowski's inequality (see Folland~\cite[p.~186]{Folland}) that
$$
\left( \E\left[|X(t) - X_{N,M,K}(t)|_\Hi^2 \right] \right)^{1/2} \leq \sum_{n=N+1}^\infty \sum_{m=M+1}^\infty \sum_{k=K+1}^\infty |h_k|_\Hi \left( \E\left[|Y_{n,m,k}(t)|^2 \right] \right)^{1/2}\,,
$$
and note furthermore that
$$
\E\left[|Y_{n,m,k}(t)|^2 \right] \leq (\E[|L_m(1)|^2])^{1/2} \int_0^t |(\Gamma(t,s)(u_n)v_m,h_k)_{\Hi}|^2 \E[|(\sigma(s),u_n)_\U|^2] \,ds\,.
$$
Given the respective ONB's, one can thus make the error induced by means of the truncated Hambit field \eqref{def:XNMK} arbitrarily small. The rate of convergence, on the other hand, is not easily derived in 
the general set-up, and requires some more structure on the Hilbert spaces to be quantified.

Sometimes it may be convenient to express the Hambit field in terms of a finite set of given "nice" vectors in $\Hi$. To this end, let
$\{\xi_1,\xi_2,\ldots,\xi_n\}$ be $n$ linearly independent elements of $\Hi$, and denote by $\Hi_n$ the subspace of $\Hi$ spanned by these.
Note that $\{\xi_i\}_{i=1}^n$ may be a subset of the basis functions of $\Hi$, but in general they are not. Introduce the projection operator 
$\mathcal{P}_n:\Hi\rightarrow\Hi_n$ defined as
\begin{equation}
\mathcal{P}_n(f)=(f,\vec{\xi}_n)_{\Hi}'H_n^{-1}\vec{\xi}_n\,,
\end{equation}
for $f\in\Hi$. Here, $(f,\vec{\xi}_n)_{\Hi}'=((f,\xi_1)_{\Hi},\ldots,(f,\xi_n)_{\Hi})'\in\R^n$, $\vec{\xi}_n$ is the vector with coordinates
$\xi_1,\xi_2,\ldots,\xi_n$, and $H_n$ the symmetric $n\times n$-matrix with coordinates $(\xi_i,\xi_j)_{\Hi}, i,j=1,\ldots,n$ assumed
to be invertible. We recall from basic functional analysis that $\mathcal{P}_n(f)$ is the element in $\Hi_n$ which minimizes the distance,
that is, $|f-\mathcal{P}_n(f)|_{\Hi}=\inf_{g\in\Hi_n}|f-g|_{\Hi}$. 
In the next Proposition we state the representation of
an Hambit field projected down on $\Hi_n$ in the Gaussian case:
\begin{proposition}
Let $L=W$ be an $\V$-valued Wiener process. Then for any $n\in\N$ there exists an $n$-dimensional standard Brownian motion $\vec{B}(t)=(B_1(t),...,B_n(t))$ 
such that
$$
\mathcal{P}_n(X(t))=\int_0^t\gamma(t,s)\,d\vec{B}(s)'H_n^{-1}\vec{\xi}_n\,,
$$
with $\gamma(t,s)$ being the square-root of the 
symmetric, positive definite stochastic $n\times n$ variance-covariance
matrix 
$$
C(t,s)=\{(\mathcal{Q}^{1/2}\Gamma(t,s)(\sigma(s))^*\xi_i,\mathcal{Q}^{1/2}\Gamma(t,s)(\sigma(s))^*\xi_j)_{\V}\}_{i,j=1}^{n}\,.
$$
\end{proposition}
\begin{proof}
By definition, we have
$$
\mathcal{P}_n(X(t))=(X(t),\vec{\xi}_n)'_{\Hi}H_n^{-1}\vec{\xi}_n\,, 
$$
which can be written as 
$$
\mathcal{P}_n(X(t))=\mathcal{T}_n(X(t))'H_n^{-1}\vec{\xi}_n\,,
$$
for the operator $\mathcal{T}_n\in L(\Hi,\R^n)$ defined by 
$$
\mathcal{T}_n(f)=((f,h_1)_{\Hi},\ldots,(f,h_n)_{\Hi})'\,.
$$
Note that for any $\vec{x}\in\R^n$, we have
$$
\mathcal{T}_n(f)'\vec{x}=(f,\mathcal{T}_n^*(\vec{x}))_{\Hi}\,,
$$
and therefore $\mathcal{T}_n^*\in L(\R^n,\Hi)$ is
$$
\mathcal{T}_n^*(\vec{x})=\vec{x}'\vec{\xi}_n\,.
$$
From Thm.~2.1 in Benth and Kr\"uhner~\cite{BK-HJM}, we obtain the existence of an $n$-dimensional Brownian motion
$\vec{B}$ such that
$$
\mathcal{T}_n(X(t))=\int_0^t\gamma(t,s)\,d\vec{B}(s)\,,
$$
for $\gamma(t,s)\in\R^{n\times n}$ where
$$
\gamma(t,s)^2=\mathcal{T}_n\Gamma(t,s)(\sigma(s))\mathcal{Q}\Gamma(t,s)(\sigma(s))^*\mathcal{T}^*_n\,.
$$
But by definition of the involved operators
$$
\mathcal{T}_n\Gamma(t,s)(\sigma(s))\mathcal{Q}\Gamma(t,s)(\sigma(s))^*\mathcal{T}^*_n(\vec{x})=C(s,t)\vec{x}\,.
$$
The matrix $C(t,s)$ is obviously symmetric by definition. Since, for any $\vec{x}\in\R^n$, we find
$$
\vec{x}^{\text{T}}C(s,t)\vec{x}=|\mathcal{Q}^{1/2}\Gamma(s,t)(\sigma(s))^*h|_{\V}^2\geq 0\,,
$$
with $h:=\sum_{i=1}^nx_ih_i\in\Hi$, positive definiteness of $C(s,t)$ follows. Thus, $C$ has a square-root and the proof is complete.
\end{proof}
In a practical situation one aims at choosing $\xi_i$ such that the elements in $C$ are easy to compute. 
We note that $\vec{\beta}_n:=H_{n}^{-1/2}\vec{\xi}_n$ is an $n$-dimensional vector of orthonormal 
basis elements of $\Hi_n$. 

\section{Hambit fields and hyperbolic SPDEs}

By a simple change of variables, one may view an Hambit field as the solution of a linear hyperbolic SPDE evaluated at the boundary.  In the present Section we analyse this connection in further detail.

To this end, let
$\widetilde{\Hi}$ be a separable Hilbert space of strongly measurable $\Hi$-valued functions on $\R_+$, and denote by $\mathcal{S}_{\xi}$ for $\xi\geq 0$ the
right-shift operator defined by $\mathcal{S}_{\xi}f=f(\xi+\cdot)$ for $f\in\widetilde{\V}$. We assume that $\{\mathcal{S}_{\xi}\}_{\xi\geq 0}$ is a
$C_0$-semigroup on $\widetilde{\Hi}$. The generator of $\mathcal{S}_{\xi}$ is seen to be $\partial_{\xi}=\partial/\partial\xi$, being a densely defined unbounded operator on $\widetilde{\Hi}$. 

Consider the SPDE
\begin{equation}
\label{abstract-ambit-spde}
dY(t)=\partial_{\xi}Y(t)\,dt+\Gamma(t+\cdot,t)(\sigma(t))\,dL(t)\,,
\end{equation}
with initial value $Y(0)\in\widetilde{\Hi}$. We suppose that for every $f\in\Hi$ and $t\in\R_+$, the mapping
\begin{equation}
\R_+\mapsto\Hi:\quad \xi\mapsto\Gamma(t+\xi,t)(\sigma(t))f\,,
\end{equation}
is an element of $\widetilde{\Hi}$ and that $\Gamma(t+\cdot,t)(\sigma(t))\in L(\V,\widetilde{\Hi})$.
Furthermore, we suppose that 
\begin{equation}
\E\left[\int_0^t\|\Gamma(s+\cdot,s)(\sigma(s))\mathcal{Q}^{1/2}\|^2_{\text{HS}}\,ds\right]<\infty
\end{equation}
which makes the stochastic integral term in \eqref{abstract-ambit-spde} well-defined. Remark that the Hilbert-Schmidt norm is with respect to linear operators from $\V$ to $\widetilde{\Hi}$ and that predictability of the integrand is ensured from the definition of a Hambit field.  

Assume the additional integrability condition on the noise term of the SPDE in \eqref{abstract-ambit-spde} , 
\begin{equation}
\E\left[\int_0^t\|\Gamma(t+\cdot,s)(\sigma(s))\mathcal{Q}^{1/2}\|^2_{\text{HS}}\,ds\right]<\infty\,.
\end{equation}
Then, by Peszat and Zabczyk~\cite[Ch. 9]{PZ} there exists a unique mild solution of 
\eqref{abstract-ambit-spde} given by
the predictable $\widetilde{\Hi}$-valued stochastic process $Y(t)$,
\begin{align}
\label{spde-mild-solution}
Y(t)&=\mathcal{S}_{t}Y(0)+\int_0^t\mathcal{S}_{t-s}\Gamma(s+\cdot,s)(\sigma(s))\,dL(s)\nonumber \\
&=\mathcal{S}_{t}Y(0)+\int_0^t\Gamma(t+\cdot,s)(\sigma(s))\,dL(s)\,.
\end{align}
We have the following result, which can be used to link $Y$ to the Hambit field $X$. 
\begin{proposition}
\label{thm:commute-L}
Suppose $\{\Phi(s)\}_{s\in\mathbb{R}_+}$ is a predictable process with values in $L(\V,\widetilde{\Hi})$ being $L$-integrable.
If $\Le\in L(\widetilde{\Hi},\Hi)$, then
$$
\Le\int_0^t\Phi(s)\,dL(s)=\int_0^t\Le\Phi(s)\,dL(s)\,.
$$
\end{proposition}
\begin{proof}
Note first that $\int_0^t\Phi(s)\,dL(s)$ takes values in $\widetilde{\Hi}$, while $\int_0^t\Le\Phi(s)\,dL(s)$ takes values in
$\Hi$ since $\Le\Phi(s)\in L(\V,\Hi)$. Moreover, since $\Le$ is a bounded operator,
$$
|\Le\Phi(s)v|_{\Hi}^2\leq \|\Le\|_{\text{op}}^2|\Phi(s)v|_{\widetilde{\Hi}}^2\,,
$$  
and
$$
\E\left[\int_0^t\|\Le\Phi(s)\mathcal{Q}^{1/2}\|_{\text{HS}}^2\,ds\right]\leq \|\Le\|_{\text{op}}^2\E\left[\int_0^t\|\Phi(s)\mathcal{Q}^{1/2}\|_{\text{HS}}^2\,ds\right]<\infty\,,
$$
by the integrability assumption on $\Phi$. Thus, $s\mapsto \Le\Phi(s)$ is $L$-integrable. 

Let $\widetilde{\Phi}$ be a simple process in $L(\V,\widetilde{\Hi})$, e.g.,
$$
\widetilde{\Phi}(s)=\sum_{k=1}^n\widetilde{\Phi}_k\mathrm{1}(s_{k-1}\leq s<s_k)\,.
$$
Then
$$
\Le\widetilde{\Phi}(s)=\sum_{k=1}^n\Le\widetilde{\Phi}_k\mathrm{1}(s_{k-1}\leq s<s_k)\,,
$$
is a simple process in $L(\V,\Hi)$, and
\begin{align*}
\Le\int_0^t\widetilde{\Phi}(s)\,dL(s)&=\Le\sum_{k=1}^n\widetilde{\Phi}_k(\Delta L(s_k))=\sum_{k=1}^n\Le\widetilde{\Phi}_k(\Delta L(s_k))=\int_0^t\Le\widetilde{\Phi}(s)\,dL(s)\,.
\end{align*}
Thus, the proposition holds for simple processes.

Let $\{\widetilde{\Phi}_n\}_{n\in\N}$ be a sequence of simple processes such that 
$$
\E\left[\int_0^t\|(\widetilde{\Phi}_n(s)-\Phi(s))\mathcal{Q}^{1/2}\|_{\text{HS}}^2\,ds\right]\rightarrow 0
$$
when $n\rightarrow\infty$. 
It holds,
\begin{align*}
\E\left[\int_0^t\|(\Le\widetilde{\Phi}_n(s)-\Le\Phi(s))\mathcal{Q}^{1/2}\|_{\text{HS}}^2\,ds\right]&\leq
\|\Le\|_{\text{op}}^2\E\left[\int_0^t\|(\widetilde{\Phi}_n(s)-\Phi(s))\mathcal{Q}^{1/2}\|_{\text{HS}}^2\,ds\right]
\end{align*}
and therefore $\{\Le\widetilde{\Phi}_n\}_{n\in\N}$ is a sequence of simple processes approximating $\Le\Phi$. 
Hence, by definition of stochastic integration, we find
$$
\int_0^t\Le\Phi(s)\,dL(s)=\lim_{n\rightarrow\infty}\int_0^t\Le\widetilde{\Phi}_n(s)\,dL(s)\,.
$$ 
As $\Le$ is a linear bounded operator, we find for any sequence $\{X_n\}_{n\in\N}$ of square integrable random variables
in $\widetilde{\Hi}$ such that $\E[|X_n-X|_{\widetilde{\Hi}}^2]\rightarrow 0$ for $X\in\widetilde{\Hi}$ when $n\rightarrow\infty$, that 
$$
\lim_{n\rightarrow\infty}\E[|\Le(X_n-X)|_{\Hi}^2]\leq\|\Le\|_{\text{op}}^2\lim_{n\rightarrow\infty}\E[|X_n-X|_{\widetilde{\Hi}}^2]=0\,.
$$
Therefore $\Le X_n$ converges to $\Le X$ in $L^2(\Omega;\Hi)$. Since, by definition,
$$
\int_0^t\Phi(s)\,dL(s)=\lim_{n\rightarrow\infty}\int_0^t\widetilde{\Phi}_n(s)\,dL(s)\,,
$$
it follows that
\begin{align*}
\Le\int_0^t\Phi(s)\,dL(s)&=\Le\lim_{n\rightarrow\infty}\int_0^t\widetilde{\Phi}_n(s)\,dL(s)
=\lim_{n\rightarrow\infty}\Le\int_0^t\widetilde{\Phi}_n(s)\,dL(s) \\
&=\lim_{n\rightarrow\infty}\int_0^t\Le\widetilde{\Phi}_n(s)\,dL(s)
=\int_0^t\Le\Phi(s)\,dL(s).
\end{align*}
The proposition follows.
\end{proof}
As a corollary, we obtain the following result:
\begin{corollary}
Assume that the evaluation map $\delta_x:\widetilde{\Hi}\rightarrow\Hi$ for $x\geq 0$ defined by $\delta_xf=f(x)$ is a
bounded linear operator, i.e., $\delta_x\in L(\widetilde{\Hi},\Hi)$ for any $x\geq 0$. If $Y(0)=0$, then, $X(t)=\delta_0Y(t)$ for $Y$ in \eqref{spde-mild-solution}. 
\end{corollary}
\begin{proof}
By Prop.~\ref{thm:commute-L} it follows,
$$
\delta_0\int_0^t\Gamma(t+\cdot,s)(\sigma(s))\,dL(s)=\int_0^t\Gamma(t,s)(\sigma(s))\,dL(s)\,,
$$
e.g., that the evaluation map commutes with the stochastic integral. 
\end{proof}
In the following Section 4 we will develop an iterative scheme for $X$ based on finite difference
approximation of the solution $Y$ of the SPDE \eqref{abstract-ambit-spde}. 

We next construct an explicit example of a space $\widetilde{\Hi}$.

\subsection{An example of $\widetilde{\Hi}$}   

We define the Filipovic space for Hilbert space-valued functions on $\R_+$. Our extension follows essentially the steps
by Filipovic~\cite{F} and rests on fundamental properties of so-called vector-valued
functions. 

Given a separable Hilbert space $\Hi$ with norm $|\cdot|_{\Hi}$ induced by the inner product denoted
$(\cdot,\cdot)_{\Hi}$. Let us recall some basic facts of vector-valued functions 
 that we shall need (see Hunter~\cite{Hunter}). 
First, a function $f:\R_+\rightarrow\Hi$ is {\it Bochner integrable}
if and only if it is {\it weakly measurable} and
$$
\int_{\R_+}|f(x)|_{\Hi}\,dx<\infty\,.
$$
Weak measurability means that $x\mapsto(f(x),g)_{\Hi}:\R_+\rightarrow\R$ is measurable for
every $g\in\Hi$. We remark that since $\Hi$ is a separable Hilbert space, weak measurability is equivalent to
strong measurability (see Hunter~\cite{Hunter}, Thm. 6.16, page 197)\footnote{{\it Strongly measurable} 
means that $f$ can be approximated by simple functions, that is,  
$f_n=\sum_{j=1}^nc_j\mathrm{1}_{E_j}$, where $\{E_j\}_{j\in\mathbb{N}}\subset\mathcal{B}_{\R_+}$ and $\{c_j\}_{j\in\mathbb{N}}\subset\Hi$, such that $|f(x)-f_n(x)|_{\Hi}\rightarrow 0$, a.e. for $x\in\R_+$ when $n\rightarrow\infty$.}. 
Let $L^1_{\text{loc}}(\R_+;\Hi)$ be the space of locally Bochner integrable functions $f:\R_+\rightarrow\Hi$. 
According to Def.~6.31, page 201 in Hunter~\cite{Hunter}, a function $f\in L_{\text{loc}}^1(\R_+;\Hi)$ 
is said to be {\it weakly differentiable} if there exists $f'\in L^1_{\text{loc}}(\R_+;\Hi)$ 
such that 
$$
\int_{\R_+}f(x)\phi'(x)\,dx=-\int_{\R_+}f'(x)\phi(x)\,dx
$$ 
for all $\phi\in C_c^{\infty}(\R_+)$. The integrals above are understood in the Bochner sense.

We are now ready to define a space of $\Hi$-valued "smooth" functions.
\begin{definition}
For a non-decreasing function $w\in C^1(\R_+)$ with $w(0)=1$, define
$$
\Hi_w=\left\{f\in L^1_{\text{loc}}(\R_+;\Hi)\,|\, \text{ there exists } f'\in L^1_{\text{loc}}(\R_+;\Hi) \text{ such that } 
\|f\|_w<\infty\right\}\,,
$$
where
$$
\|f\|_w^2=|f(0)|_{\Hi}^2+\int_0^{\infty}w(x)|f'(x)|^2_{\Hi}\,dx\,.
$$
\end{definition} 
Denote by $\langle\cdot,\cdot\rangle_w$ the inner product
$$
\langle f,g\rangle_w=(f(0),g(0))_{\Hi}+\int_0^{\infty}w(x)(f'(x),g'(x))_{\Hi}\,dx\,,
$$
for $f,g\in\Hi_w$, and observe that $\|f\|_w^2=\langle f,f\rangle_w$. 
\begin{proposition}
\label{prop:H-sep-hilbert}
$(\Hi_w, \|\cdot\|_w)$ is a separable Hilbert space.
\end{proposition}
\begin{proof}
The proof adapts the arguments of Thm. 5.1.1 in Filipovic~\cite{F} to Hilbert-valued functions. 
We include the details here for the convenience of the reader. 

Observe that $L^2(\R_+;\Hi)$ is a Hilbert space, and so is $\Hi\times L^2(\R_+;\Hi)$ with norm 
$\|\cdot\|_*^2:=|\cdot|_{\Hi}^2+\|\cdot\|^2_{L^2(\R_+;\Hi)}$. Define the linear operator 
$T:\Hi_w\rightarrow\Hi\times L^2(\R_+;\Hi)$ by
\begin{equation}
\label{def-T-operator}
Tf=(f(0),f'\sqrt{w})\,.
\end{equation}
$T$ is isometric, since
\begin{align*}
\|Tf\|_*^2&=|f(0)|^2_{\Hi}+\int_0^{\infty}|\sqrt{w(x)}f'(x)|^2_{\Hi}\,dx=\|f\|_w^2\,.
\end{align*}
We claim that its inverse is the operator $S:\Hi\times L^2(\R_+;\Hi)\rightarrow\Hi_w$ defined as
$$
S(u,h)(x):=u+\int_0^xh(y)w^{-1/2}(y)\,dy\,.
$$
First, since $h\in L^2(\R_+;\Hi)$ and $w^{-1/2}(y)\leq 1$ due to $w$ being non-decreasing and $w(0)=1$, we find that the integral is well-defined in the Bochner sense. It holds,
$$
T(S(u,h))=(S(u,h)(0),S(u,h)'\sqrt{w})=(u,hw^{-1/2}w^{1/2})=(u,h)\,,
$$
where we have applied Thms. 6.32 (page 201) and 6.35 (page 203) in Hunter~\cite{Hunter}. Furthermore,
$$
S(Tf)(x)=S(f(0),f'\sqrt{w})(x)=f(0)+\int_0^{x}f'(y)\sqrt{w}(y)w^{-1/2}(y)\,dy=f(0)+\int_0^{x}f'(y)\,dy=f(x)\,,
$$
where we used Thm 6.35 (page 203) in Hunter~\cite{Hunter} in the last equality. 
Hence, $S=T^{-1}$, and $\Hi_w$ is isomorphic to the Hilbert
space $\Hi\times L^2(\R_+;\Hi)$ implying that $\Hi_w$ is a complete inner product space, i.e., a Hilbert space. 

$\Hi$ is assumed separable, which means that for any $f\in L^2(\R_+;\Hi)$
$$
f(x)=\sum_{k=1}^{\infty}\langle f(x),e_k\rangle_{\Hi} e_k
$$
for $a.e.\, x\in\R_+$ for the ONB $\{e_k\}_{k\in\mathbb{N}}\subset\Hi$. We have that 
$f_k:=\langle f(\cdot),e_k\rangle_{\Hi}\in L^2(\R_+;\R)$ since by Schwartz' inequality
$$
\int_0^{\infty}|f_k(x)|^2\,dx=\int_0^{\infty}|\langle f(x),e_k\rangle_{\Hi}|^2\,dx\leq\int_0^{\infty}|f(x)|_{\Hi}^2\,dx|e_k|_{\Hi}^2
=\int_0^{\infty}|f(x)|_{\Hi}^2\,dx<\infty\,.
$$
But since $L^2(\R_+;\R)$ is separable, we find for an ONB $\{h_n\}_{n\in\mathbb{N}}\subset L^2(\R_+;\R)$
$$
f_k(x)=\sum_{n=1}^{\infty}(f_k,h_n)_{L^2}h_n(x)\,.
$$ 
But then $\{h_n\otimes e_k\}_{n,k\in\mathbb{N}}$ is an ONB of 
$L^2(\R_+;\Hi)$. This shows that $L^2(\R_+;\Hi)$ is separable, and hence $\Hi\times L^2(\R_+;\Hi)$ is
separable as well. By the isomorphism $T$, we can therefore conclude the separability of $\Hi_w$. 
The proof is complete.
\end{proof}
The next Lemma provides us with a fundamental theorem of calculus on $\Hi_w$:
\begin{lemma}
\label{lemma:Hw_FTC}
Assume $w^{-1}\in L^1(\R_+)$. Then for any $f\in\Hi_w$, $f'\in L^1(\R_+;\Hi)$,
$\|f'\|_{L^1(\R_+;\Hi)}\leq c\|f\|_w$, and 
$$
f(x+t)-f(x)=\int_x^{x+t}f'(y)\,dy\,,
$$
for every $x\in\R_+$ and $t\geq 0$. The constant $c$ is given by $c^2:=\int_0^{\infty}w^{-1}(x)\,dx$.  
\end{lemma}
\begin{proof}
For $f\in\Hi_w$, we find by the Cauchy-Schwartz inequality,
\begin{align*}
\int_0^{\infty}|f'(x)|_{\Hi}\,dx&=\int_0^{\infty}w^{-1/2}(x)w^{1/2}(x)|f'(x)|_{\Hi}\,dx \\
&\leq (\int_0^{\infty}w^{-1}(x)\,dx)^{1/2}(\int_0^{\infty}w(x)|f'(x)|_{\Hi}^2\,dx)^{1/2} \\
&\leq(\int_0^{\infty}w^{-1}(x)\,dx)^{1/2}\|f\|_w<\infty\,.
\end{align*}
Hence, $f'\in L^1(\R_+;\Hi)$ and the norm estimate follows. But then, by Thm. 6.35 (page 203) in 
Hunter~\cite{Hunter} yields the fundamental theorem of calculus.
\end{proof}
This result also tells us that any element of $\Hi_w$ is absolutely 
continuous, and, in particular, continuous.
 
Introduce now the shift semigroup  $(\mathcal{S}_t)_{t\geq 0}$ on $\Hi_w$ defined as
\begin{equation}
\mathcal{S}_tf:=f(\cdot+t)\,.
\end{equation}
The next Lemma shows uniform boundedness of $\mathcal{S}_t$ on $\Hi_w$. 
\begin{lemma}
\label{lemma:shift-uniformly-bounded}
Suppose that $w^{-1}\in L^1(\R_+)$, Then $\mathcal{S}_t$ is uniformly bounded with
$\|\mathcal{S}_t\|_{\text{op}}\leq\sqrt{2(1+c^2)}$. Here, $c$ is the positive constant
defined in Lemma~\ref{lemma:Hw_FTC}. 
\end{lemma}
\begin{proof}
Since, by Lemma~\ref{lemma:Hw_FTC} and Thm. 6.32 in Hunter~\cite{Hunter}, $(\mathcal{S}_tf)'=f'(\cdot+t)$, we find that
\begin{align*}
\|\mathcal{S}_tf\|_w^2&=|f(t)|_{\Hi}^2+\int_0^{\infty}w(x)|f'(x+t)|^2_{\Hi}\,dx \\
&\leq |f(0)+\int_0^tf'(y)\,dy|^2_{\Hi}+\int_t^{\infty}w(y-t)|f'(y)|^2_{\Hi}\,dy \\
&\leq 2|f(0)|^2_{\Hi}+2|\int_0^tf'(y)\,dy|_{\Hi}^2+\int_0^{\infty}w(y)|f'(y)|^2_{\Hi}\,dy \\
&\leq 2|f(0)|_{\Hi}^2+2(\int_0^t|f'(y)|_{\Hi}\,dy)^2+\int_0^{\infty}w(y)|f'(y)|^2_{\Hi}\,dy\,.
\end{align*}
In the first inequality we applied Lemma~\ref{lemma:Hw_FTC}, 
while in the second we applied an elementary inequality and the
monotonicity of $w$. Finally, in the last estimation step we used the norm inequality for Bochner integrals. 
Hence, again appealing to the monotonicity of $w$,
$$
\|\mathcal{S}_tf\|_w^2\leq 2(1+c^2)\|f\|_w^2
$$
and the proof is complete.
\end{proof}
Next, we study continuity properties of the shift semigroup $\mathcal{S}_t$. To this end,
let  
\begin{equation}
\label{domain-deriv-op}
\mathcal{D}:=\left\{f\in\Hi_w\,|\,f'\in\Hi_w\right\}\,,
\end{equation}
where we note that $\text{Dom}(\partial_x)=\mathcal{D}$ and $\partial_x$ being the derivative operator.
\begin{lemma}
If $w^{-1}\in L^1(\R_+)$, the shift operator $\mathcal{S}_t$ is strongly continuous on $\mathcal{H}_w$. 
\end{lemma}
\begin{proof}
We first show strong continuity on $\mathcal{D}$ defined in \eqref{domain-deriv-op}. 
Indeed, for $f\in\Hi$ we have by the Lemma~\ref{lemma:Hw_FTC} above
$$
f(t)-f(0)=\int_0^tf'(y)\,dy\,.
$$
Moreover, if $f'\in\Hi_w$, then the same Lemma yields,
$$
f'(x+t)-f'(x)=\int_x^{x+t}f''(y)\,dy=t\int_0^1f''(x+st)\,ds\,.
$$
Also, we have that $\mathcal{S}_t f=f(\cdot+t)\in\Hi_w$ is weakly differentiable 
(see proof of Lemma~\ref{lemma:Hw_FTC}).  
Thus, for $f\in\mathcal{D}$ we find from the the norm inequality for Bochner integrals and 
Cauchy-Schwartz' inequality, 
\begin{align*}
\|\mathcal{S}_tf-f\|_w^2&=|f(t)-f(0)|_{\Hi}^2+\int_0^{\infty}w(x)|f'(x+t)-f'(x)|^2_{\Hi}\,dx \\
&=|\int_0^tf'(y)\,dy|^2_{\Hi}+\int_0^{\infty}w(x)t^2|\int_0^1f''(x+st)\,ds|_{\Hi}^2\,dx \\
&\leq(\int_0^t|f'(y)|^2_{\Hi}\,dy)^2+t^2\int_0^{\infty}w(x)(\int_0^1|f''(x+st)|^2_{\Hi}\,ds)^2\,dx \\
&\leq(\int_0^t|f'(y)|_{\Hi}^2\,dy)^2+t^2\int_0^{\infty}\int_0^1w(x)|f''(x+st)|^4_{\Hi}\,ds\,dx \\
&\leq(\int_0^t|f'(y)|_{\Hi}^2\,dy)^2+t^2\int_0^1\|\mathcal{S}_{st}f'\|_w^4\,ds\,.
\end{align*}
The second integral is finite as $\mathcal{S}_t$ is a uniformly bounded operator on $\Hi_w$ from 
Lemma~\ref{lemma:Hw_FTC}. Thus, letting $t\downarrow 0$, we get that $\|\mathcal{S}_tf-f\|_w\rightarrow 0$, showing
strong continuity on $\mathcal{D}$. 

By appealing to a density argument for $\mathcal{D}$ in $\Hi_w$, we can conclude that
$\mathcal{S}_t$ is strongly continuous on $\Hi_w$:  Introduce the subspace (following Filipovic~\cite{F}, page 77)
$$
\mathcal{D}_0=\{f\in C^2(\R_+;\Hi)\,|\,f'\in C^1_c(\R_+;\Hi)\}\,,
$$
where $C^2(\R_+;\Hi)$ denotes the twice continuously strongly differentiable
functions and $C^1_c(\R_+;\Hi)$ functions with compact support being once continuously strongly
differentiable. 
Prop.~6.29 in Hunter~\cite{Hunter} ensures that $C_c^1(\R_+;\Hi)$ is dense in $L^2(\R_+;\Hi)$. For $f\in\Hi_w$, 
let $\{h_n\}_n\subset C_c^1(\R_+;\Hi)$ be an approximating sequence of $f'\sqrt{w}\in L^2(\R_+;\Hi)$. Define
$f_n:=T^{-1}(f(0),h_n)$ for the operator $T$ defined in \eqref{def-T-operator}. We have that $f_n\in\mathcal{D}_0$ 
and $\|f_n-f\|_w\rightarrow 0$ as $n\rightarrow\infty$ because $T$ is an isomorphism (see proof of
Prop.~\ref{prop:H-sep-hilbert}). This shows that $\mathcal{D}_0$ is dense in $\mathcal{H}_w$.  

Thus, for $f,g\in\Hi_w$, the triangle inequality along with the uniform boundedness of $\mathcal{S}_t$, yield,
\begin{align*}
\|\mathcal{S}_tf-f\|_w&\leq \|\mathcal{S}_t(f-g)\|_w+\|\mathcal{S}_tg-g\|_w+\|g-f\|_w \\
&\leq\sqrt{2(1+c^2)}\|f-g\|_w+\|f-g\|_w+\|\mathcal{S}_tg-g\|_w\,.
\end{align*}
But, since $\mathcal{D}_0$ is dense in $\Hi_w$, we choose $g\in\mathcal{D}_0$ such that
$\|f-g\|_w\leq \epsilon/2(1+\sqrt{2(1+c^2)}$. By strong continuity of $\mathcal{S}_t$ on $\mathcal{D}$ we choose
$t$ such that $\|\mathcal{S}_tg-g\|_w\leq \epsilon/2$. Then, $\mathcal{S}_t$ is strongly continuous on $\Hi_w$. 
The proof is complete.
\end{proof}
We conclude that $\mathcal{S}_t$ is a $C_0$-semigroup on $\Hi_w$ with a generator $\partial_x$ being 
defined on $\mathcal{D}$, a dense subset of $\mathcal{H}_w$.  
 
Introduce the evaluation map $\delta_x;\Hi_w\rightarrow\Hi$ for $x\in\R_+$ as
$\delta_xf:=f(x)$ for $f\in\Hi_w$. We prove that $\delta_x$ is a bounded linear operator:
\begin{lemma}
Suppose that $w^{-1}\in L^1(\R_+)$. Then $|\delta_xf|_{\Hi}\leq K\|f\|_w$ for a positive constant $K$ given by
$K^2=2\max(1,\int_0^{\infty}w^{-1}(y)\,dy)$.
\end{lemma}
\begin{proof}
For $f\in\Hi_w$ it holds by Lemma~\ref{lemma:Hw_FTC} that 
$$
\delta_xf=f(x)=f(0)+\int_0^xf'(y)\,dy\,.
$$
But then by Bochner's norm inequality and Cauchy-Schwartz' inequality,
\begin{align*}
|f(x)|_{\Hi}^2&\leq 2|f(0)|_{\Hi}^2+2|\int_0^xf'(y)\,dy|^2_{\Hi} \\
&\leq2|f(0)|_{\Hi}^2+2(\int_0^x|f'(y)|_{\Hi}\,dy)^2 \\
&\leq 2|f(0)|_{\Hi}^2+2\int_0^{\infty}w^{-1}(y)\,dy\int_0^{\infty}w(y)|f'(y)|^2_{\Hi}\,dy\,.
\end{align*}
This concludes the proof.
\end{proof}

We end this Subsection with some results on linear functionals on $\Hi$ and $\Hi_w$. To this end, 
let $H_w$ be the classical Filipovic space (which can be obtained by selecting $\Hi=\R$ in the definition
of $\Hi_w$ above). The norm is denoted  by $|\cdot|_w$. We have the following proposition:
\begin{proposition}
For $\mathcal{L}\in\Hi^*$ and $g\in\Hi_w$, the real-valued function $x\mapsto\mathcal{L}(g(x))$ on $\R_+$ is
an element of $H_w$.
\end{proposition} 
\begin{proof}
Recall that if $g\in\Hi_w$, then $g(x)\in\Hi$ for any $x\in\R_+$, and thus $\mathcal{L}(g(\cdot))$ is a real-valued measurable
function on $\R_+$ which is locally integrable. As $g\in\Hi_w$ it is weakly differentiable,
$$
g(x)=g(0)+\int_0^xg'(y)\,dy\,,
$$
and by properties of the Bochner integral
$$
\mathcal{L}(g(x))=\mathcal{L}(g(0))+\int_0^x\mathcal{L}(g'(y))\,dy\,.
$$ 
Hence, $x\mapsto\mathcal{L}(g(x))$ is weakly differentiable and $\partial_x(\mathcal{L}(g(x)))=\mathcal{L}(g'(x))$,
for $\partial_x$ being the differential operator. Thus,
\begin{align*}
|\mathcal{L}(g(\cdot))|_w^2&=|\mathcal{L}(g(0))|^2+\int_0^{\infty}w(y)|\mathcal{L}(g'(y))|^2\,dy \\
&\leq\|\mathcal{L}\|_{\text{op}}^2|g(0)|_{\Hi}^2+\|\mathcal{L}\|_{\text{op}}^2\int_0^{\infty}w(y)|g'(y)|_{\Hi}^2\,dy \\
&=\|\mathcal{L}\|^2_{\text{op}}|g|_{\Hi}^2<\infty\,.
\end{align*}
The result follows.
\end{proof}
Note that we can write $\mathcal{L}(g(x))=\mathcal{L}\circ\delta_x (g)$, and that $\mathcal{L}\circ\delta_x\in\Hi_w^*$,
whenever $\mathcal{L}\in\Hi^*$. 
This means that there exists a unique $\ell_x\in\Hi_w$ such that
$$
\mathcal{L}(g(x))=\mathcal{L}\circ\delta_x(g)=\langle g,\ell_x\rangle_w\,.
$$
We can characterize $\ell_x$:
\begin{proposition}
Assume $w^{-1}\in L^1(\R_+)$. It holds $\mathcal{L}(g(x))=\langle g,\ell_x\rangle_w$ for $\ell_x(\cdot)\in\Hi_w$ where
$$
\ell_x(\cdot)=\mathcal{L}^*(h_x(\cdot))
$$
for $y\mapsto h_x(y)=1+\int_0^{x\wedge y}w^{-1}(z)\,dz\in H_w$. 
\end{proposition}
\begin{proof}
From Lemma 5.3.1 in Filipovic̃\cite{F}, 
$$
\mathcal{L}(g(x))=\bar{\delta}_x(\mathcal{L}(g(\cdot)))=(\mathcal{L}(g(\cdot)),h_x)_w
$$
where $\bar{\delta}_x$ is the evaluation map on $H_w$. Hence,
\begin{align*}
\mathcal{L}(g(x))&=(\mathcal{L}(g(\cdot)),h_x)_w \\
&=\mathcal{L}(g(0))1+\int_0^{\infty}w(y)\mathcal{L}(g'(y))h_x'(y)\,dy \\
&=(g(0),\mathcal{L}^*1)_{\Hi}+\int_0^{\infty}w(y)(g'(y),\mathcal{L}^*(h_x'(y)))_{\Hi}\,dy\,.
\end{align*}
We find that $\ell_x'(y)=\mathcal{L}^*(h_x'(y))$ by linearity of $\mathcal{L}^*$ and the fundamental theorem
of calculus. Noting that
$h_x(0)=1$, the proof follows. 
\end{proof}

\section{A finite difference scheme}
\label{FD}

This section presents a finite difference scheme for approximating solutions of a slightly generalized version of the hyperbolic SPDE \eqref{abstract-ambit-spde}. More specifically, we consider the hyperbolic SPDE
set in $\widetilde{\Hi}$
\begin{equation}
\label{spde:gen}
dY(t) = \partial_\xi Y(t)\,dt + \beta(t)\,dL(t)\,, 
\end{equation}
with given initial value $Y(0) = Y_0 \in \widetilde \Hi$. Here, $\beta(t)\in L(\V,\widetilde{\Hi}$)
is predictable and such that
$$
\E\left[\int_0^t\|\beta(s)\mathcal{Q}^{1/2}\|_{\text{HS}}^2\,ds\right]<\infty\,.
$$
For the special case of Hambit processes, we choose $\beta(t)=\Gamma(t+\cdot,t)(\sigma(t))$. However, in
this section, we simplify the notation by considering a general stochastic integrand $\beta$. Suppose
in addition that 
$$
\E\left[\int_0^t\|\mathcal{S}_{t-s}\beta(s)\mathcal{Q}^{1/2}\|_{\text{HS}}^2\,ds\right]<\infty\,,
$$
then by Peszat and Zabczyk~\cite[Ch. 6]{PZ}, the SPDE~\eqref{spde:gen} possesses a mild solution
$$
Y(t)=\mathcal{S}_tY_0+\int_0^t\mathcal{S}_{t-s}\beta(s)\,dL(s)\,.
$$
In what follows, we can easily include a drift in the SPDE above, but we refrain from doing so reduce 
notation and technicalities.



Let $\Delta x > 0$ and $\Delta t > 0$ denote the discrete steps in space and time respectively, and
set $t_n=n\Delta t$, $x_j=j\Delta x$ for $n=0,\ldots,N$ and $j=0,\ldots,J$ for some $J,N \in \N$. We aim 
at introducing an approximation $\widetilde{Y}^n$ of $Y$ at time $t_n$ of the form
\begin{equation}
\widetilde{Y}^n=\sum_{j=0}^{J-1}\left\{\frac{x-x_j}{\Delta x}(y_{j+1}^n-y_j^n)+y_j^n\right\}\mathbf{1}_{[x_j,x_{j+1})}(\cdot)\,,
\end{equation}
for $x\leq x_J$ and $\widetilde{Y}^n(x)=y_J^n$ for $x>x_J$. Here, $\{y_j^n\}_{j=0,\ldots J}\subset\Hi$. We assume that $\widetilde{Y}^n\in\widetilde{\Hi}$, and remark that in the case $\widetilde{\Hi}=\Hi_{w}$ this assumption holds since the weak derivative of $\widetilde{Y}^n$ in 
that case is piecewise constant and zero outside $x>x_J$. 
It is convenient to think of
$$
y_j^n \approx \delta_{j\Delta x}Y(n\Delta t)\,,
$$ 
that is, $\delta_{x_j}\widetilde{Y}^n$ approximates the sampled solution of \eqref{spde:gen} at the point $(n\Delta t, j\Delta x)$. Here,
we recall the evaluation operator $\delta_x\in L(\widetilde{\Hi},\Hi)$ introduced in the previous section.
For the initial value $Y_0$, we introduce the approximation
\begin{equation}
\widetilde{Y}_0:=\sum_{j=0}^{J-1}\left\{\frac{x-x_j}{\Delta x}(\delta_{x_{j+1}}Y_0-\delta_{x_j}Y_0)+\delta_{x_j}Y_0\right\}\mathbf{1}_{[x_j,x_{j+1})}(\cdot)\,,
\end{equation}
for $x\leq x_J$ and $\widetilde{Y}_0(x)=\delta_{x_J}Y_0$ for $x>x_J$.  
This is indeed a linear interpolation of $Y_0$ as an element in $\widetilde{\Hi}$. We assume 
$\widetilde{Y}_0\in\widetilde{\Hi}$, and obviously let $y_j^0:=\delta_{x_j}\widetilde{Y}_0=\delta_{x_j}Y_0$.  
Since $\beta(t)\in L(\V,\widetilde{\Hi})$,
$\delta_x\beta(t)\in L(\V,\Hi)$. 
As we shall see, we need a particular approximation of $\beta(t)$, denoted by $\widetilde{\beta}(t)$ 
and given as (for $x\leq x_J$)
\begin{equation}
\widetilde{\beta}(t)=\sum_{j=0}^{J-1}\left\{\frac{x-x_j}{\Delta x}(\delta_{x_{j+1}}\beta(t)-\delta_{x_j}\beta(t))+\delta_{x_j}\beta(t)\right\}\mathbf{1}_{[x_j,x_{j+1})}(\cdot)\,,
\end{equation}
and $\widetilde{\beta}(t)(x)=\delta_{x_J}\beta(t)$ for $x>x_J$. 
Thus, we sample the operator $\beta(t)\in L(\V,\widetilde{\Hi})$ into a linear combination of operators
$\delta_{x_j}\beta(t)\in L(\V,\Hi), j=0,\ldots,J$. We see that $x\mapsto\widetilde{\beta}(t)(x)$
is a function from $\R_+$ into $L(\V,\Hi)$. We therefore define 
$\widetilde{\beta}(t)\in L(\V,\widetilde{\Hi})$ by 
\begin{equation}
\widetilde{\beta}(t)(f)=\sum_{j=0}^{J-1}\left\{\frac{x-x_j}{\Delta x}(\delta_{x_{j+1}}\beta(t)(f)-\delta_{x_j}\beta(t)(f))+\delta_{x_j}\beta(t)(f)\right\}\mathbf{1}_{[x_j,x_{j+1})}(\cdot)\,,
\end{equation}
for $f\in\V$, with $x\leq x_J$. When $x>x_J$, we let $\widetilde{\beta}(t)(f)(x)=\delta_{x_J}\beta(t)(f)$. Since $\delta_{x_j}\beta(t)(f)\in\Hi$, $\widetilde{\beta}(t)(f)$ is a function from $\R_+$
into $\Hi$. We {\it assume} that $\widetilde{\beta}(t)(f)\in\widetilde{\Hi}$ from now on, and remark
that when $\widetilde{\Hi}=\widetilde{\Hi}_w$, this assumption is fulfilled since we have a piecewise
constant weak derivative which is zero outside $x_J$. 

To derive a recursive scheme for $y_j^n$ in $n$, we use finite difference approximations of
the SPDE \eqref{spde:gen}, thus using $dY(t) \approx Y(t+\Delta t) - Y(t)$, $dt \approx \Delta t$, $dL(t) \approx L(t+\Delta t) - L(t)$ and  $\partial_\xi Y(t) \approx (Y(t)(\cdot + \Delta x) - Y(t))/\Delta x$ in
\eqref{spde:gen} to find the {\it finite difference scheme}
\begin{equation}
\label{def:FD}
y_j^{n+1} = \lambda y_{j+1}^n + (1-\lambda)y_j^n + \beta_j^n(\Delta L^n)\,,
\end{equation}
where $\lambda = \Delta t / \Delta x$, 
$\beta_j^n = \delta_{x_j}\beta(t_n)$ and $\Delta L^n = L(t_{n+1}) - L(t_n)$. 

We note that the finite difference scheme \eqref{def:FD} is a Hilbert space generalization of a scheme proposed and analysed by Benth and Eyjolfsson~\cite{BE}. In that paper a numerical approximation of
real-valued VMV processes based on a scheme for a hyperbolic SPDE was introduced, analogous to 
the case we study here. Our infinite dimensional approach and analysis that follows are inspired by
Benth and Eyjolfsson~\cite{BE}. Notice that the information in the finite difference scheme in \eqref{def:FD} flows to the left as time progresses. Hence, for a given time $n\Delta t$, the scheme will provide values for $y_j^{n+1}$, $j=0,\ldots,J-1$ for the next time step. As we wish to study our approximation $\widetilde{Y}^n$ for $n=0,1,\ldots N$ 
and $x\leq x_J$, we can adjust our finite differencing to be made for suitably large choices of grid points
in space $x$ initially, so that at terminal time $N\Delta t$ we have a computation of $y_j^{N}$ for
all $j=0,\ldots,J$. Indeed, this is the same as letting $J$ be depending on the time step $n$. We 
refrain from going into technical details on the practicalities here, but refer to Benth and Eyjolfsson~\cite{BE} for more discussion.

As in the case of a finite difference scheme for the standard advection partial differential equation, one needs some constraints on the discrete steps, i.e. $(\Delta x,\Delta t)$, to guarantee its stability. The stability condition of Courant, Friedrichs, and Lewy (the CFL condition, see \cite{CFL}) is needed to ensure the stability of our finite difference scheme \eqref{def:FD}. In our case this translates into the necessary constraint 
\begin{equation}
\label{stable}
\Delta t \leq \Delta x,
\end{equation}
which we assume to hold.

Given our Hilbert space $\widetilde \Hi$ of $\Hi$-valued functions on $\R_+$  it will be convenient for our analysis to define the following family of bounded linear operators on $\widetilde \Hi$. Given positive $\Delta x > 0$ and $\Delta t > 0$ corresponding to the steps of the finite difference scheme in space and time respectively consider the family $\{\mathcal{T}_{\Delta x,\Delta t}\}_{\Delta x > 0, \Delta t > 0}$ which is defined by 
\begin{equation}
\label{def:T}
\mathcal T_{\Delta x,\Delta t} = \mathcal{I} + \Delta t\frac{\mathcal{S}_{\Delta x} - \mathcal I}{\Delta x},
\end{equation}
for all $\Delta x > 0, \Delta t > 0$, where $\mathcal I$ denotes the identity operator on $\widetilde{\Hi}$. 
\begin{lemma}
\label{lem:itscheme}
For given steps $\Delta x > 0$ in space and $\Delta t > 0$ in time, $\widetilde{Y}^n$ admits the representation
\begin{equation}
\widetilde{Y}^n=\mathcal{T}^n\widetilde{Y}_0+\sum_{i=0}^{n-1}\mathcal{T}^{n-1-i}\widetilde{\beta}^i(\Delta L^i)\,,
\end{equation}
for $n=0,\ldots,N$. 
Here, $\mathcal T := \mathcal T_{\Delta x, \Delta t}$ is defined by \eqref{def:T}, and we use the conventions that  $\mathcal T^n = \mathcal T^{\circ n}$ denotes the composition of the operator 
$\mathcal T$ with itself $n$ times, and $\mathcal T^0 =\mathcal I$.
\end{lemma}
\begin{proof}
We prove the result by induction. It clearly holds for $n=0$, since then 
$\widetilde{Y}^0=\widetilde{Y}_0=\mathcal{T}^0\widetilde{Y}_0$. Next, suppose that it holds
for $n\in\N$. Assume that $x\in[x_j,x_{j+1})$ for a given $j\in\N$, $j\leq J$. Then, 
$x+\Delta x\in[x_{j+1},x_{j+2})$, and we find
\begin{align*}
\delta_x\mathcal T\widetilde{Y}^n&=\delta_x\mathcal I\widetilde{Y}^n+\lambda\delta_x(\mathcal{S}_{\Delta x}-\mathcal I)\widetilde{Y}^n  \\
&=y_j^n+\lambda(y_{j+1}^n-y_j^n)+\frac{x-x_j}{\Delta x}(y_{j+1}^n+\lambda(y_{j+2}^n-y_{j+1}^n)) -\frac{x-x_j}{\Delta x}(y_j^n+\lambda(y_{j+1}^n-y_j^n))\,.
\end{align*}
But by the finite difference scheme \eqref{def:FD}, it follows
\begin{align*}
\delta_x\mathcal T\widetilde{Y}^n&=y_j^{n+1}-\beta_j^n(\Delta L^n)+\frac{x-x_j}{\Delta x}
(y_{j+1}^{n+1}-\beta_{j+1}^n(\Delta L^n))-\frac{x-x_j}{\Delta x}(y_j^{n+1}-\beta_j^n(\Delta L^n)) \\
&=y_j^{n+1}+\frac{x-x_j}{\Delta x}(y_{j+1}^{n+1}-y_{j}^{n+1}) \\
&\qquad-\left(\beta_j^n(\Delta L^n)+\frac{x-x_j}{\Delta x}(\beta_{j+1}^n(\Delta L^n)-\beta_j^n(\Delta L^n))\right)\,.
\end{align*}
By invoking the definition of $\widetilde{\beta}(t)$ and noting that $x$ can be chosen arbitrary, 
$$
\widetilde{Y}^{n+1}=\mathcal T\widetilde{Y}^n+\widetilde{\beta}^n(\Delta L^n)\,.
$$
From the induction hypothesis, we then find
$$
\widetilde{Y}^{n+1}=\mathcal{T}^{n+1}\widetilde{Y}_0+\mathcal{T}\sum_{i=0}^{n-1}\mathcal{T}^{n-1-i}\widetilde{\beta}^i(\Delta L^i)+\widetilde{\beta}^n(\Delta L^n)=\mathcal{T}^{n+1}\widetilde{Y}_0+\sum_{i=0}^{n}\mathcal{T}^{n-i}\widetilde{\beta}^i(\Delta L^i)\,.
$$
This completes the proof.
\end{proof}

The above lemma characterizes the finite difference scheme \eqref{def:FD} for a given discretization 
as the sum of two entities which, under appropriate conditions, will converge to their corresponding parts in the mild solution of  \eqref{spde:gen} as we consider finer and finer partitions in time and space. More precisely we will employ the fact that the composed operator $\mathcal T^n$, where $\mathcal T = \mathcal T_{\Delta x, \Delta t}$ is defined by \eqref{def:T}, converges to the left shift operator $\mathcal S_{t_n}$ as we consider finer and finer partitions in first time and then space. 

Let us take a closer look on the family \eqref{def:T} of operators. The following lemma will be employed later for proving a convergence result on the finite difference scheme.
\begin{lemma}
\label{lem:Tconv}
Suppose $\zeta$ is an $\widetilde{\Hi}$-valued random variable satisfying the Lipschitz condition 
$$
\E[\|(\mathcal{S}_x\zeta - \mathcal{S}_y\zeta)\mathcal Q^{1/2}\|_{\text{HS}}^2] \leq C|x-y|^2
$$
for all $x,y \geq 0$ where $C > 0$ is a constant. Then 
$$
\E[\|(\mathcal T^m \zeta - \mathcal S_t\zeta)\mathcal Q^{1/2} \|_{\text{HS}}^2] \leq Ct(\Delta x - \Delta t).
$$ 
where $\mathcal T$ is defined in \eqref{def:T} with $\Delta t = t/m$ and $\Delta t \leq \Delta x$, for all $x \geq 0$, $t > 0$ and  $m \geq 1$.
\end{lemma}
\begin{proof}
Let $\lambda = \Delta t/\Delta x$ and suppose first that $\lambda = 1$, then clearly $\mathcal T = \mathcal S_{\Delta x}$ and $\mathcal T^m = \mathcal S_t$. Now suppose that $\lambda < 1$, and observe that by the binomial theorem it holds that
\begin{align*}
\mathcal T^m\zeta &= (1-\lambda)^m\left(\mathcal I + \frac{\lambda}{1-\lambda}\mathcal S_{\Delta x}\right)^m\zeta 
= \sum_{k=0}^m {m \choose k} \lambda^k (1-\lambda)^{m-k}\mathcal{S}_{k\Delta x}\zeta\,.
\end{align*} 
It follows by the triangle inequality that
\begin{align*}
\|(\mathcal T^m \zeta - \mathcal S_t\zeta)\mathcal Q^{1/2} \|_{\text{HS}}^2 &= \left\| \left(\sum_{k=0}^m {m \choose k} \lambda^k (1-\lambda)^{m-k} (\mathcal{S}_{k\Delta x}\zeta-\mathcal S_t\zeta) \right)\mathcal Q^{1/2} \right\|_{\text{HS}}^2 \\
&\leq \left| \sum_{k=0}^m {m \choose k} \lambda^k (1-\lambda)^{m-k} \left\|(\mathcal{S}_{k\Delta x}\zeta-\mathcal S_t\zeta)\mathcal Q^{1/2} \right\|_{\text{HS}} \right|^2 \\
&\leq \sum_{k=0}^m {m \choose k} \lambda^k (1-\lambda)^{m-k} \left\|(\mathcal{S}_{k\Delta x}\zeta - \mathcal S_t\zeta)\mathcal Q^{1/2} \right\|_{\text{HS}}^2\,.
\end{align*}
In the last step we applied the Cauchy-Schwarz inequality. Finally, we employ the Lipschitz condition 
on $\zeta$ to derive,
\begin{align*}
\E[\|(\mathcal T^m \zeta - \mathcal S_t\zeta)\mathcal Q^{1/2} \|_{\text{HS}}^2] 
&\leq C\sum_{k=0}^m {m \choose k} \lambda^k (1-\lambda)^{m-k} |k\Delta x - t|^2\,. 
\end{align*}
Observing that a binomial random variable $Z$, with parameters $(m,\lambda)$ has expected value $m\lambda$ and variance $m\lambda(1-\lambda)$, it is easy to deduce that the random variable $\Delta xZ$ has expected value $t$ and variance $t(\Delta x -\Delta t)$. Hence,
$$
\sum_{k=0}^m {m \choose k} \lambda^k (1-\lambda)^{m-k} |k\Delta x - t|^2=t(\Delta x - \Delta t)\,.
$$
This concludes the proof.
\end{proof}
We can apply the same type of argument to derive the error induced by approximating
$\mathcal{S}_tY_0$ by $\mathcal{T}^{m}\widetilde{Y}_0$:
\begin{lemma}
\label{lemma:conv-initial}
Assume for $x,y\in\R_+$ that $|\mathcal{S}_xY_0-\mathcal{S}_yY_0|_{\widetilde{\Hi}}\leq C_0|x-y|$
for some positive constant $C_0$. Then, 
$$
|\mathcal{T}^{m}\widetilde{Y}_0-\mathcal{S}_tY_0|_{\widetilde{\Hi}}\leq C_0\sqrt{t(\Delta x-\Delta t)}
 + \sup_{u \leq t}\|\mathcal{S}_{u}\|_{\text{op}} |\widetilde{Y}_0-Y_0|_{\widetilde{\Hi}}\,,
$$
where $\mathcal T$ is defined in \eqref{def:T} with $\Delta t = t/m$ and $\Delta t \leq \Delta x$, for all $x \geq 0$, $t > 0$ and  $m \geq 1$.
\end{lemma}
\begin{proof}
By the triangle inequality
$$
|\mathcal{T}^{m}\widetilde{Y}_0-\mathcal{S}_tY_0|_{\widetilde{\Hi}}\leq 
|\mathcal{T}^{m}\widetilde{Y}_0-\mathcal{T}^{m}Y_0|_{\widetilde{\Hi}}
+|\mathcal{T}^{m}Y_0-\mathcal{S}_tY_0|_{\widetilde{\Hi}}\,.
$$
For the second term on the right hand side, using the Lipschitz assumption on $Y_0$,  we can repeat the argument in the proof Lemma~\ref{lem:Tconv} for the norm $|\cdot|_{\widetilde{\Hi}}$ instead of $\|\cdot\|_{\text{HS}}$ to obtain
$$
|\mathcal{T}^{m}Y_0-\mathcal{S}_tY_0|_{\widetilde{\Hi}}\leq C_0\sqrt{t(\Delta x-\Delta t)}\,.
$$
For the first term, we find
$$
|\mathcal{T}^{m}\widetilde{Y}_0-\mathcal{T}^{m}Y_0|_{\widetilde{\Hi}}\leq
\|\mathcal{T}^m\|_{\text{op}}|\widetilde{Y}_0-Y_0|_{\widetilde{\Hi}}\,.
$$
Now suppose first that $\lambda = 1$, then clearly $\mathcal{T} = \mathcal{S}_{\Delta x}$ and $\mathcal{T}^m = \mathcal{S}_t$. If however $\lambda < 1$, then 
$$
\|\mathcal{T}^m\|_{\text{op}} = \sup\{| \mathcal{T}^m f |_{\widetilde{\Hi}}  : f \in \widetilde{\Hi}, |f|_{\widetilde{\Hi}}=1\}, 
$$
and we may apply the binomial theorem to obtain
\begin{align*}
\mathcal T^mf &= (1-\lambda)^m\left(\mathcal I + \frac{\lambda}{1-\lambda}\mathcal S_{\Delta x}\right)^mf 
= \sum_{k=0}^m {m \choose k} \lambda^k (1-\lambda)^{m-k}\mathcal{S}_{k\Delta x}f\,,
\end{align*}
so it follows by the triangle inequality that 
\begin{equation}\label{Tm:ineq}
\|\mathcal{T}^m\|_{\text{op}} \leq \max_{0 \leq k \leq m}\|\mathcal{S}_{k\Delta x}\|_{\text{op}} \leq \sup_{u \leq t}\|\mathcal{S}_{u}\|_{\text{op}}\,.
\end{equation}
This completes the proof.
\end{proof} 
In general the operator norm of a $C_0$-semigroup grows at most exponentially with time,
so that we find $\sup_{u\leq t}\|\mathcal{S}_u\|_{\text{op}}\leq c_1\exp(c_2 t)$ for positive 
constants $c_1,c_2$. If $\widetilde{\Hi}=\Hi_w$ with $w^{-1}\in L^1(\R_+)$, the shift semigroup
$\mathcal{S}_t$ is uniformly bounded by Lemma~\ref{lemma:shift-uniformly-bounded}, and moreover
$\sup_{u\leq t}\|\mathcal{S}_u\|_{\text{op}}\leq \sqrt{2(1+c^2)}$ for $c^2=\int_0^{\infty}w^{-1}(x)\,dx$.

\begin{proposition}
\label{prop:Conv}
Assume that for $s,u,x,y\in\R_+$, 
$$
|\mathcal{S}_xY_0-\mathcal{S}_yY_0|_{\widetilde{\Hi}} \leq C_0|x-y|\,, 
$$
$$
\E\left[\| (\mathcal{S}_x\beta(s) - \mathcal{S}_y\beta(s))\mathcal{Q}^{1/2}\|_{\text{HS}}^2\right] \leq C|x-y|^2\,,
$$
and 
$$
\E\left[\| (\beta(s) -\beta(u))\mathcal{Q}^{1/2}\|_{\text{HS}}^2\right] \leq C|s-u|^2\,,
$$
for positive constants $C_0, C$. 
Then, for $t_n = n\Delta t$ and $x_j = j\Delta x$, $n,j \geq 0$, it holds that
\begin{align*}
\E\left[|\widetilde{Y}^N-Y(t_N)|_{\widetilde{\Hi}}^2\right] 
&\leq 4t(C_0^2+2Ct)(\Delta x-\Delta t)+8 C t \left(1+\frac13\sup_{u \leq t}\|\mathcal{S}_u\|_{\text{op}}^2\right)(\Delta t)^2 \\
&\qquad + 4\sup_{u \leq t}\|\mathcal{S}_{u}\|_{\text{op}}^2 \E\left[|\widetilde{Y}_0-Y_0|_{\widetilde{\Hi}}^2\right] \\
&\qquad + 8 t \sup_{u \leq t}\|\mathcal{S}_u\|_{\text{op}}^2 \max_{0 \leq i \leq N-1} \E\left[\|(\widetilde{\beta}^i
-\beta^i)\mathcal{Q}^{1/2}\|_{\text{HS}}^2\right]\,.
\end{align*}
where $\Delta t = t/N$ and $\Delta t \leq \Delta x$, for all $x \geq 0$, $t > 0$ and  $N \geq 1$.
\end{proposition}
\begin{proof}
Since $L$ is square integrable it holds by the It\^o isometry that 
\begin{align*}
\E&\left[\left|\sum_{i=0}^{N-1} \mathcal T^{N-1-i}\widetilde{\beta}^i(\Delta L^i)-\sum_{i=0}^{N-1} \mathcal S_{t-t_{i+1}} \beta^i(\Delta L^i)\right|_{\widetilde{\Hi}}^2\right] \\
&\qquad=\E\left[\left|\int_{0}^t \left(\sum_{i=0}^{N-1} (\mathcal T^{N-1-i}\widetilde{\beta}^i - \mathcal S_{t-t_{i+1}}\beta^i)\mathbf{1}_{[t_i,t_{i+1})}(s)\right)\,dL(s)\right|_{\widetilde{\Hi}}^2\right] \\
&\qquad=\E\left[\int_{0}^t\|\sum_{i=0}^{N-1}(\mathcal T^{N-1-i}\widetilde{\beta}^i-\mathcal S_{t-t_{i+1}}\beta^i)\mathbf{1}_{[t_i,t_{i+1})}(s)\mathcal Q^{1/2}\|_{\text{HS}}^2\, ds\right] \\
&\qquad= \sum_{i=0}^{N-1} \E[ \|(\mathcal T^{N-1-i}\widetilde{\beta}^i-\mathcal S_{t-t_{i+1}}\beta^i)\mathcal Q^{1/2}\|_{\text{HS}}^2 ]\, \Delta t \,.
\end{align*}
Adding and subtracting $\mathcal{T}^{N-1-i}\beta^i$ and applying the elementary inequality
$(x+y)^2\leq 2x^2+2y^2$ yields, 
\begin{align*}
\E&\left[\left|\sum_{i=0}^{N-1} \mathcal T^{N-1-i}\widetilde{\beta}^i(\Delta L^i)-\sum_{i=0}^{N-1} \mathcal S_{t-t_{i+1}} \beta^i(\Delta L^i)\right|_{\widetilde{\Hi}}^2\right] \\
&\qquad\leq 2\sum_{i=0}^{N-1}\E\left[\|(\mathcal{T}^{N-1-i}\widetilde{\beta}^i-\mathcal{T}^{N-1-i}\beta^i)\mathcal{Q}^{1/2}\|_{\text{HS}}^2\right]\,\Delta t \\
&\qquad\qquad+2\sum_{i=0}^{N-1}\E\left[\|(\mathcal{T}^{N-1-i}\beta^i-\mathcal{S}_{t-t_{i+1}}\beta^i)\mathcal{Q}^{1/2}\|_{\text{HS}}^2\right]\,\Delta t \,.
\end{align*}
We estimate the second term by appealing to Lemma~\ref{lem:Tconv}, whereas the first term is
majorized by using the inequality \eqref{Tm:ineq}. Hence, 
\begin{align*}
\E&\left[\left|\sum_{i=0}^{N-1} \mathcal T^{N-1-i}\widetilde{\beta}^i(\Delta L^i)-\sum_{i=0}^{N-1} \mathcal S_{t-t_{i+1}} \beta^i(\Delta L^i)\right|_{\widetilde{\Hi}}^2\right] \\
&\qquad\leq 2\sup_{u \leq t}\|\mathcal{S}_u\|_{\text{op}}^2 \sum_{i=0}^{N-1}\E\left[\|(\widetilde{\beta}^i
-\beta^i)\mathcal{Q}^{1/2}\|_{\text{HS}}^2\right]\,\Delta t+2Ct^2(\Delta x-\Delta t) \\
&\qquad \leq 2t \sup_{u \leq t}\|\mathcal{S}_u\|_{\text{op}}^2 \max_{0 \leq i \leq N-1} \E\left[\|(\widetilde{\beta}^i
-\beta^i)\mathcal{Q}^{1/2}\|_{\text{HS}}^2\right]\ + 2Ct^2(\Delta x-\Delta t)\,.
\end{align*}
Furthermore by Lipschitz continuity of $\beta$ and the It\^o isometry, 
\begin{align*}
\E&\left[\left|\sum_{i=0}^{N-1}\mathcal S_{t-t_{i+1}} \beta^i(\Delta L^i) - \int_{0}^t \mathcal S_{t-s} \beta(s)\,dL(s)\right|_{\widetilde{\Hi}}^2\right] \\
&=  \E\left[\left|\int_{0}^t \left(\sum_{i=0}^{N-1} \mathcal S_{t-t_{i+1}} \beta^i 1_{[t_i,t_{i+1})}(s) - \mathcal S_{t-s}\beta(s) \right)\,dL(s)\right|_{\widetilde{\Hi}}^2\right] \\
&= \sum_{i=0}^{N-1}\E\left[\int_{t_i}^{t_{i+1}} \left\| \left( \mathcal S_{t-t_{i+1}}\beta^i - \mathcal S_{t-s}\beta(s)\right)\mathcal Q^{1/2}\right\|_{\text{HS}}^2\,ds\right] \\
&\leq 2\sum_{i=0}^{N-1}\E\left[\int_{t_i}^{t_{i+1}}\left( \left\| \left( \mathcal S_{t-t_{i+1}}\beta^i - \mathcal S_{t-s}\beta^i \right)\mathcal Q^{1/2}\right\|_{\text{HS}}^2 +\left\| \mathcal S_{t-s}(\beta^i - \beta(s)) \mathcal Q^{1/2}\right\|_{\text{HS}}^2 \right)\right] \,ds \\
&\leq 2tC(\Delta t)^2+2\sup_{u \leq t}\|\mathcal S_{u}\|_{\text{op}}^2
\sum_{i=0}^{N-1}\int_{t_i}^{t_{i+1}}\E\left[ \left\|(\beta(t_i) - \beta(s)) \mathcal Q^{1/2}\right\|_{\text{HS}}^2 \right]\,ds \\
&\leq 2Ct\left(1 + \frac13\sup_{u \leq t}\|\mathcal S_{u}\|_{\text{op}}^2\right)(\Delta t)^2 \,. 
\end{align*}
Putting the above inequalities together, we obtain 
\begin{align*}
\E&\left[\left|\sum_{i=0}^{N-1} \mathcal T^{N-1-i}\widetilde{\beta}^i(\Delta L^i) - \int_{0}^t \mathcal S_{t-s} \beta(s)\, dL(s)\right|_{\widetilde{\Hi}}^2 \right] \\
&\qquad\leq 2\E\left[\left|\sum_{i=0}^{N-1}\mathcal S_{t-t_{i+1}} \beta^i(\Delta L^i) - \int_{0}^t \mathcal S_{t-s} \beta(s)\,dL(s)\right|_{\widetilde{\Hi}}^2\right] \\
&\qquad\qquad+2\E\left[\left|\sum_{i=0}^{N-1} \mathcal T^{N-1-i}\widetilde{\beta}^i(\Delta L^i)-\sum_{i=0}^{N-1} \mathcal S_{t-t_{i+1}} \beta^i(\Delta L^i)\right|_{\widetilde{\Hi}}^2\right] \\
&\qquad\leq 4 t \sup_{u \leq t}\|\mathcal{S}_u\|_{\text{op}}^2 \max_{0 \leq i \leq N-1} \E\left[\|(\widetilde{\beta}^i
-\beta^i)\mathcal{Q}^{1/2}\|_{\text{HS}}^2\right]  \\
&\qquad\qquad  + 4 C t \left(1+\frac13\sup_{u \leq t}\|\mathcal{S}_u\|_{\text{op}}^2\right)(\Delta t)^2 + 4Ct^2(\Delta x-\Delta t) \,.
\end{align*}
The proof is completed after invoking Lemma~\ref{lemma:conv-initial}.
\end{proof}
Recall that $\beta(t):=\Gamma(t+\cdot,t)(\sigma(t)$ in the case of a Hambit field, for which we see that 
$\beta(s)-\beta(u)=\mathcal{S}_{s-u}\beta(u)-\mathcal{S}_0\beta(u)$ for $s\geq u\geq 0$. Hence, the two 
Lipschitz conditions on  $\beta$ in the Proposition above collapse into one, namely
$$
\E[\|\Gamma(s+x,s)(\sigma(s))-\Gamma(s+y,s)(\sigma(s))Q^{1/2}\|_{HS}^2]\leq C|x-y|^2\,,
$$ 
for all $x,y,s\in\R_+$. Thus, if the operator $\Gamma$ is Lipschitz continuous in its first argument,
the conditions on $\beta$ are fulfilled. The condition on $Y_0$ is trivially satisfied for Hambit fields
as $Y_0=0$ in that case.

As we have already touched upon it is not trivial to express a given Hambit field in terms of a certain finite set of vectors in $\Hi$. It is however the case according to  Proposition \ref{prop:ambit-lss} that for given ONB's in the Hilbert spaces $\U,\V$ and $\Hi$, a general Hambit field can be represented as a countable sum of real-valued VMV processes scaled by the ONB vectors in $\Hi$. Although it is difficult to say anything in general about the rate at which that sum converges, it is clear that it can be truncated, and thus our finite difference scheme \eqref{def:FD} can be implemented at least in an approximative manner, for a given Hambit field which fulfills the conditions stated in Proposition \ref{prop:ambit-lss}.

Now let us elucidate what the above convergence result means for the Hilbert space $\widetilde{\Hi} = \Hi_w$, which we introduced in the previous section. 
Note that,
\begin{align*}
\left\| (\widetilde{\beta}^i-\beta^i) \mathcal Q^{1/2}\right\|_{\text{HS}} 
&\leq \left\|\mathcal Q^{1/2}\right\|_{\text{HS}} \| \widetilde{\beta}^i-\beta^i\|_{\text{op}} =  \left\|\mathcal Q^{1/2}\right\|_{\text{HS}} \sup_{|f|_{\V}=1}| (\widetilde{\beta}^i-\beta^i)(f) |_{\widetilde{\Hi}}\,.
\end{align*}
Therefore, the convergence of 
$$
\max_{0 \leq i \leq N-1} \E\left[\|(\widetilde{\beta}^i
-\beta^i)\mathcal{Q}^{1/2}\|_{\text{HS}}^2\right]
$$
depends on the convergence of 
$$
\|(\widetilde \beta^i-\beta^i)(f)\|_w = |(\widetilde \beta^i-\beta^i)f(0)|_\Hi^2 + \int_0^\infty w(x)|(\widetilde \beta^i-\beta^i)(f)'(x)|_\Hi^2 dx,
$$
in $L^2(\Omega)$, where $|f|_\V = 1$, as we consider finer and finer partitions. We remark that if $f \in \V$ and $x \in [x_j,x_{j+1})$, then we may express the weak derivative above as
$$
(\widetilde \beta^i-\beta^i)(f)'(x) = \frac{\delta_{x_{j+1}} - \delta_{x_j}}{\Delta x} \beta^i f - \beta^i f'(x).
$$
That is, the right hand side is equal to the difference between a $\Hi$-valued finite difference approximation and its corresponding weak derivative evaluated at $x \in [x_j,x_{j+1})$. So the convergence of the scheme depends on the convergence of the above finite difference approximation in $\Hi$.

\end{document}